\documentclass[12pt,reqno]{amsproc}
\usepackage[margin = 1in]{geometry}

\usepackage{comment}
\usepackage[small,bf]{caption}
\usepackage{nicefrac}
\usepackage[usenames, dvipsnames]{color}
\usepackage{hyperref}
\hypersetup
{
    colorlinks,	%
    citecolor=blue,%
    filecolor=black,%
    linkcolor=blue,%
    urlcolor=blue
}
\usepackage{amsmath}
\usepackage{amsthm}
\usepackage{amssymb}
\usepackage{amsfonts}
\usepackage{mathabx}
\usepackage{eulervm, xfrac}
\usepackage[sans]{dsfont}
\usepackage{mathdots}
\usepackage{mathrsfs}
\usepackage{rotating, setspace}
\DeclareMathAlphabet{\mathbfsf}{\encodingdefault}{\sfdefault}{bx}{n}
\usepackage[all]{xy}
\usepackage{tikz-cd}
\usepackage{enumerate}
\usepackage[capitalise]{cleveref}

\usepackage{thmtools,thm-restate}

\DeclareFontShape{OT1}{cmr}{bx}{sc}{<-> cmbcsc10}{}

\usepackage[colorinlistoftodos]{todonotes}

\newcommand{\define}[1]{\textbf{#1}}

\newcommand{\cosheaf}[1]{\widehat{#1}}
\theoremstyle{definition}
\newtheorem{thm}{Theorem}[section]
\newtheorem{lem}[thm]{Lemma}
\newtheorem{defn}[thm]{Definition}
\newtheorem{cor}[thm]{Corollary}
\newtheorem{prop}[thm]{Proposition}

\newtheorem{ex}[thm]{Example}

\newtheorem{rmk}[thm]{Remark}

\newcommand{\cJ}{\mathcal{J}}
\newcommand{\cM}{\mathcal{M}}

\newcommand{\calU}{\mathcal{U}}

\newcommand{\covU}{\mathcal{U}}
\newcommand{\cU}{\mathcal{U}}

\newcommand{\aat}{\mathbfsf{A}}

\newcommand{\bat}{\mathbfsf{B}}

\newcommand{\cat}{\mathbfsf{C}}

\newcommand{\dat}{\mathbfsf{D}}

\newcommand{\Vect}{\mathbfsf{Vect}}
\newcommand{\vect}{\mathbfsf{vect}}

\newcommand{\Fun}{\mathbfsf{Fun}}

\newcommand{\Int}{\mathbfsf{Int}}
\newcommand{\Open}{\mathbfsf{Open}}
\newcommand{\Down}{\mathbfsf{Down}}
\newcommand{\Up}{\mathbfsf{Up}}
\newcommand{\CoShv}{\mathbfsf{CoShv}}

\newcommand{\GridO}{\mathbfsf{GridO}}
\newcommand{\Cell}{\mathbfsf{Cell}}
\newcommand{\Trans}{\mathbfsf{Trans}}

\newcommand{\op}{\mathbfsf{op}}

\newcommand{\UU}{\mathbfsf{U}}
\newcommand{\DD}{\mathbfsf{D}}

\newcommand{\X}{\mathbb{X}}
\newcommand{\R}{\mathbb{R}}

\newcommand{\Z}{\mathbb{Z}}
\newcommand{\field}{\Bbbk}

\newcommand{\Lan}{\mathsf{Lan}}
\newcommand{\Ran}{\mathsf{Ran}}

\newcommand{\id}{\text{id}}

\newcommand{\pP}{\mathcal{P}}
\newcommand{\pQ}{\mathcal{Q}}

\newcommand{\pZ}{\mathcal{Z}}
\newcommand{\pW}{\mathcal{W}}
\newcommand{\topX}{\mathcal{X}}
\newcommand{\topY}{\mathcal{Y}}
\newcommand{\pL}{\mathcal{L}}

\newcommand{\st}{\text{star}}

\newcommand{\e}{\epsilon}

\newcommand{\squigrightarrow}{\rightsquigarrow}

\usepackage[backend = bibtex, url=false,sorting=nyt,style=numeric-comp,maxbibnames = 99]{biblatex}
\bibliography{refs}

\makeatletter
\def\l@subsection{\@tocline{2}{0pt}{2.5pc}{5pc}{}}
\def\l@subsubsection{\@tocline{2}{0pt}{5pc}{7.5pc}{}}
\makeatother
\begin{document}

\title{A Relative Theory of Interleavings}

\begin{center}
\begin{abstract}
	The interleaving distance, although originally developed for persistent homology, has been generalized to measure the distance between functors modeled on many posets or even small categories.
  Existing theories require that such a poset have a superlinear family of translations or a similar structure.
  However, many posets of interest to topological data analysis, such as zig-zag posets and the face relation poset of a cell-complex, do not admit interesting translations, and consequently don't admit a nice theory of interleavings.
  In this paper we show how one can side-step this limitation by providing a general theory where one maps to a poset that does admit interesting translations, such as the lattice of down sets, and then defines interleavings relative to this map.
  Part of our theory includes a rigorous notion of discretization or ``pixelization'' of poset modules, which in turn we use for interleaving inference.
  We provide an approximation condition that in the setting of lattices gives rise to two possible pixelizations, both of which are guaranteed to be close in the interleaving distance. Finally, we conclude by considering interleaving inference for cosheaves over a metric space and give an explicit description of interleavings over a grid structure on Euclidean space.
\end{abstract}
\end{center}

\author{Magnus Bakke Botnan, Justin Curry, Elizabeth Munch}
\maketitle

\section{Introduction}

This paper builds upon the work of Bubenik, de Silva, and Scott in~\cite{Bubenik2014a} by reconsidering the question of how to define a distance between functors modeled on an arbitrary poset $\pP$.
Those authors were motivated, as we are, by the desire to compare filtrations of a space that are indexed by a poset, by considering them through the lens of persistent homology.
Although the literature on persistent homology is vast, with independent origins in Frosini~\cite{Frosini1990distance,Frosini1992},
Barannikov~\cite{Barannikov1994},
Robins~\cite{robins1999towards},
Edelsbrunner, Letscher and Zomorodian~\cite{Edelsbrunner2000, Edelsbrunner2002},
it was the study of the stability of persistent homology~\cite{Cohen-Steiner2007} that led to the \emph{interleaving distance}~\cite{Chazal2009b} and our first prototype for defining distances between such functors.

One of the several approaches considered in~\cite{Bubenik2014a} to generalizing the interleaving distance was to assume the existence of a \emph{super-linear family of translations} on a poset.
Deferring the concept of a \emph{super-linear} family for now, we review the notion of a translation on a poset in order to show why it is deficient for the examples we wish to consider.

\begin{defn}\label{defn:translation}
A \define{translation} on a poset $\pP$ is a map $T\colon\pP \to \pP$ that is monotone, i.e.~$p\leq q$ implies $T(p) \leq T(q)$, and satisfies the identity $p\leq T(p)$ for all $p\in \pP$. Alternatively, if one views a poset as a category, where the objects are elements of $\pP$ and there is a unique morphism from $p\to q$ if and only if $p\leq q$, then a translation is equivalently described as a functor $T\colon\pP \to \pP$ that admits a natural transformation from the identity, i.e.~$\eta\colon \id_{\pP} \Rightarrow T$.
\end{defn}

Unfortunately, for a poset of the following ``zig-zag'' form, where arrows point towards elements that are higher in the partial order,
\begin{center}
\begin{equation*}
\begin{tikzpicture}[scale=1.5][baseline= (a).base]
\node[scale=0.7] (a) at (0,0){
 \begin{tikzcd}
\bullet&  & \bullet  &  & \bullet  &  & \bullet  &  & \bullet  &  & \bullet  &  & \bullet \\
 & \bullet\ar[ur]\ar[ul]  &   & \bullet\ar[ur]\ar[ul] &  & \bullet\ar[ur]\ar[ul] &   &\bullet\ar[ur]\ar[ul]  &   & \bullet\ar[ur]\ar[ul] &   & \bullet\ar[ur]\ar[ul]&
\end{tikzcd}};
\end{tikzpicture}
\label{eq:zigzag}
\end{equation*}
\end{center}
there are no translations other than the identity.
To see why, consider the following smaller poset, where we've labelled the elements for concreteness:
\begin{center}
\begin{equation*}
\begin{tikzpicture}[scale=1.5][baseline= (a).base]
\node[scale=1] (a) at (0,0){
 \begin{tikzcd}
 a  &  & c \\
 & b \ar[ur]\ar[ul]  &
\end{tikzcd}};
\end{tikzpicture}
\end{equation*}
\end{center}
The condition that $p\leq T(p)$ for all $p$, implies that $T(a)=a$ and $T(c)=c$.
This appears to give some choice as to where to send $b$, but monotonicity rules this out.
Indeed if $T(b)=c$, then the implication $b \leq a \Rightarrow T(b) \leq T(a)$ is contradicted because $T(b)=c\nleq a=T(a)$.

However, the above types of posets cannot be ignored.
Topological data analysis is replete with examples where functors modeled on posets (also called $\pP$-modules) are of the above form.
These examples include, but are not limited to, zig-zag persistence~\cite{Carlsson2009a,Carlsson2010,Milosavljevic2011,botnan2018algebraic}, Reeb graphs~\cite{deSilva2016}, circle-valued persistence~\cite{burghelea2013topological} and cellular sheaf theory
~\cite{Curry2014, Curry2015a, Munch2016}.
Consequently, to calculate distances between these algebraic-topological summaries of parametrized data, a theory of interleavings suited to zig-zags as well as more general posets is necessary.

The solution that we advance in this paper is that posets such as the above zig-zag poset naturally embed into larger posets that \emph{do} admit non-trivial translations, where the theory of~\cite{Bubenik2014a} can be applied.
Such a solution was already outlined in~\cite{botnan2018algebraic}, for the special case of zig-zag posets, by embedding such posets into $\R^2$, but here we motivate the development of a more general \emph{intrinsic} interleaving theory that uses a thickening structure on the lattice of down sets $\Down(\pP)$ to serve as our poset with non-trivial translations.
This motivation and theory is the content of \cref{sec:philosophical-interleavings}.

Moreover, for the sake of theoretical clarity, we have decided to build up a general \emph{relative theory of interleavings}, which allows us to define an interleaving theory over $\pP$, as long as we have a map of posets $f\colon\pP\to\pQ$ where $\pQ$ has a super-linear family of translations.
This is done by using the operations of pushforward and pullback, which we need to define how to \emph{shift} modules over $\pP$ (\cref{defn:shift-relative-to-f}).
This allows us to define relative interleavings (\cref{defn:relative-weak-interleaving}), which in turn leads to subtle question: Does our notion of distance differ from the distance gotten by pushing forward to $\pQ$ and working entirely with interleavings over $\pQ$? Our first main result, \cref{thm:extend-restrict-interleavings}, shows that it does not: the pushforward operation $f_*$ is an isometry onto its image.

The relative theory of interleavings we introduce has the added benefit of providing a general framework for \emph{approximating} functors modeled on a poset and performing interleaving \emph{inference}.
We do this by picking up one of the original ideas from~\cite{Chazal2009b} of discretizing---or \emph{pixelizing}---persistence modules by viewing $\pP$ as a ``discrete'' poset that is embedded, via $f$, inside of a ``continuous'' poset $\pQ$.
For us (and~\cite{brown2019probabilistic}) the pixelization is simply the iteration of the pullback and pushforward operations, i.e.~$f_*f^*M$. However we show that, in fact, there are two notions of pixelization that one can consider.
By using a simple triangle inequality argument (\cref{lem:interleaving-distortion}), we show that one can infer the interleaving distance over the poset $\pQ$ by using interleavings over $\pP$ as long as one can bound the difference between a module and its pixelization; which is provided by our $\delta$-approximation condition in \cref{defn:delta-approx}.
Our relative interleaving theory, which takes a nicer form when we assume that $\pP=\pL$ is a lattice, fits into the distortion result of \cref{thm:delta-distortion}: When $f\colon\pL \to \pQ$ is a $\delta$-approximation, the pullback operation $f^*$ distorts interleaving distances by at most $2\delta$.

The theory of lattices is especially well suited to studying open sets in a topological space and we conclude the paper in \cref{sec:mapper} with an approximation theory for cosheaves, which takes much of its inspiration from~\cite{Munch2016}.
When we put a grid structure on $\R^n$, we show an expected result: when the spacing of the grid is bounded by $\delta$, we can calculate interleavings of cosheaves using discrete offsets of cells instead of continuously growing open sets.
In this sense, our paper provides a theory of numerical analysis for interleavings of cosheaves over $\R^n$ with guaranteed tolerance bounds.

\section{Modules over General Posets}\label{sec:modules}

Much of our inspiration for this paper comes from persistent homology.
The usual story one tells about persistence is that it gives a way of assigning continuous ``shape summaries'' to finite data.
For example, given a finite subset $X$ of $\R^n$, commonly called a \emph{point cloud}, one considers an associated one-parameter family of spaces, which can be thought of as thickenings of the point cloud.
 $$X=\{x_i\}_{i=1}^N\subset \R^n \qquad \text{begets} \qquad  X_r=\cup_{i=1}^N B(x_i,r) \qquad \text{for every} \qquad r\geq 0.$$
Note that for any two radii $r \leq r'$, we have the natural inclusion $X_r \subseteq X_{r'}$ of spaces.
By applying various functorial topological lenses we can study how the topology of the point cloud changes across multiple scales indexed by $r$.
If we use homology with field coefficients as a lens, then we obtain a collection of functors, indexed by integers $i\geq 0$, associated to the point cloud:
$$PH_i(X)\colon (\R,\leq) \to \Vect \qquad r\leq r' \rightsquigarrow H_i(X_r) \to H_i(X_{r'}).$$

The above functors are called \define{persistence modules} in the literature because for every pair of related numbers $r \leq r' \in \R$ we can ask which homological features \emph{persist} along the inclusion $X_r\subseteq X_{r'}$.
We note, just as~\cite{Bubenik2014} does, that such a notion makes sense for \emph{any} partially ordered set $\pP$, making the term ``persistence module'' very general.
We avail ourselves of the following language.

\begin{defn}[$\pP$-modules]\label{defn:P-module}
Let $\pP$ be a partially ordered set.
A functor $M\colon\pP \to \Vect$ is called a \define{$\pP$-module} or a \define{module over $\pP$}.
As a reminder, such a functor can be described concretely as an assigment to each element $p\in\pP$ a vector space $M(p)$ and a linear map $M(p \leq q)\colon M(p) \to M(q)$ for every pair of comparable elements $p\leq q$.
We refer to the maps $M(p\leq q)$ as being ``internal'' to the module $M$.
Reading composition from right to left, we require that the collection of maps internal to the module $M$ satisfy the relation
\[
	M(p\leq r) = M(q\leq r) \circ M(p \leq q) \qquad \forall \quad p\leq q \leq r.
\]

A \define{map of $\pP$-modules} $\psi\colon M \to N$ is simply a natural transformation of functors.
Said concretely, this is a collection of linear maps $\psi(p)\colon M(p) \to N(p)$ that are consistent with the internal maps, i.e.
$$\psi(q)\circ M(p\leq q) = N(p\leq q) \circ \psi(p).$$

With the above definitions in hand, we note that the collection of all $\pP$-modules defines a category, written $\Fun(\pP,\Vect)$ or $\Vect^{\pP}$.
We note that any category $\cat$ can be used instead of $\Vect$ and the same definition and language makes sense.
The category of all $\pP$-modules valued in $\cat$ is written $\Fun(\pP,\cat)$ or $\cat^{\pP}$.
Unless otherwise noted, we will usually assume that $\cat=\Vect$.
\end{defn}

One might wonder just how complicated of a poset $\pP$ or module $M$ might arise in applications. In our next example, we outline a general source for $\pP$-modules.

\begin{ex}[Modules over a Simplicial Complex]\label{ex:modules-over-a-simplicial-complex}
  Any simplicial complex $K$ is a collection of subsets of a vertex set $V$ that is closed under restriction.
  Consequently, any simplicial complex $K$ has the structure of a poset, where $\sigma\leq \tau$ if and only if the vertices of $\sigma$ are a subset of the vertices of $\tau$.
  Let $\Cell(K):=K^{\op}$ be the opposite poset of $K$.

  The reason we consider the opposite poset of $K$ is that there is a monotone map $\Cell(K) \to \Open(|K|)$ that takes any simplex $\sigma$ to its open star $\st |\sigma|$ in the geometric realization of $K$, written $|K|$.
  Recall that the open star is the union of the interiors of geometrically realized simplices whose face includes $\sigma$, i.e.
  \[
    \st |\sigma|=\cup_{\sigma \leq_K \tau} |\tau|.
  \]
  Notice that if $\sigma \leq \tau$ in $K$ then $\st |\tau| \subseteq \st |\sigma|$, so this map really is order-reversing when viewed from $K$.

  The following setup now occurs quite naturally: Given any topological space $X$ and a map $f\colon X \to |K|$ we naturally get a collection of $\Cell(K)$-modules
  \[
  F_i \colon \Cell(X) \to \Vect \qquad \text{where} \qquad \sigma \quad \squigrightarrow \quad H_i(f^{-1}(\st\,|\sigma|))
  \]
  called the \define{cellular Leray cosheaves} associated to $f$.
  The cellular Leray sheaves are defined dually by using cohomology instead of homology.
\end{ex}

\subsection{Up Sets, Down Sets and Open Sets}

The following notions from the theory of partially ordered sets will be essential.

\begin{defn}
  \label{defn:downset}
  Given a poset $(\pP,\leq)$, a subset $D\subseteq \pP$ is called a \define{down set} if the following implication holds
  \[
    \text{if} \qquad q\in D \subseteq \pP \qquad \text{and} \qquad p\leq q \qquad \text{then} \qquad p\in D.
  \]
  A \define{principal down set} is any down set of the form $D_p=\{q \in\pP \mid q\leq p\}$.
  The collection of all down sets in a poset $\pP$ is denoted by $\Down(\pP)$.
  We note that associated to any subset $S\subseteq \pP$ there is an associated down set that it generates by considering $\DD(S)=\cup_{p\in S} D_p$.

  We can, of course, dualize the above notions. We say $U\subseteq \pP$ is an \define{up set} if whenever $p\in U$ and $p\leq q$, then $q\in U$.
  Principal up sets, the up set $\UU(S)$ associated to a subset $S$ and the collection of all up sets $\Up(\pP)$ are all defined analogously.
\end{defn}

We now provide some examples of down sets and up sets, which expand upon~\cref{ex:modules-over-a-simplicial-complex}.

\begin{ex}
  Associated to any set $V$ is the collection of all its subsets, written $P(V)$.
  Notice that $P(V)$ has the structure of a poset given by containment of subsets, i.e. $\sigma\leq \tau$ if and only if $\sigma\subseteq \tau$.
  A simplicial complex $K$ on $V$ is nothing more nor less than a down set in $P(V)$.
  The ``closed under restriction'' condition of a simplicial complex is exactly the down set condition.
  The complete simplex on $V$ is then an example of a principal down set.
  The set of simplices in $K$ that are incident to a simplex $
  \sigma$ forms a principal up set in $K$.
\end{ex}

One of the reasons for considering down sets and up sets is that they naturally form a topology on any poset $\pP$.
In fact, order-preserving maps between posets are exactly the maps that are continuous with respect to this topology.

\begin{defn}\label{defn:Alex-topology}
The \define{Alexandrov topology} on a poset $\pP$ is the topology where open sets are down sets.
The principal down sets $\{D_p\}$ serve as a basis for this topology.
\end{defn}

Of course, one could just as well use up sets to define a topology on $\pP$.
The only reason to prefer one over the other is that the map
\[
  \iota\colon \pP \hookrightarrow \Down(\pP) \qquad p \mapsto \iota(p) := D_p
\]
is order preserving, whereas the corresponding map for up sets is order reversing.
Consequently, every poset $\pP$ embeds into the lattice\footnote{Actually, the poset of open sets forms a \define{spatial frame}. A \define{frame} is a lattice that satisfies the infinite distributive property.} $\Down(\pP)$, where the meet and join operations correspond to the intersection and union operations for open sets, respectively.
The interested reader should be aware that this observation is the jumping off point for several interesting duality theorems: Birkhoff's theorem\cite{birkhoff1937}, Priestley's Theorem~\cite{priestley1970}, and Stone duality~\cite{TaylorP:fofct}.

One of the main ways in which we sidestep the difficulties of~\cite{Bubenik2014a} is to work with translations over $\Down(\pP)$, which can be non-trivial and lead to an interesting theory of interleavings over zig-zag and other type posets.
This requires that we be able to take modules defined over $\pP$ and extend them to modules over $\Down(\pP)$.
This is the subject of the next section.

\subsection{Pushforward and Pullback of Modules Over Posets}

For all of this paper $f\colon\pP \to \pQ$ will denote a map of posets.
A map of posets is always assumed to be order-preserving, which is another way of saying monotone. In other words, the following implication is true:
\[
  \text{If} \qquad p\leq_{\pP} q \qquad \text{then} \qquad f(p) \leq_{\pQ} f(q).
\]
The reader should convince themselves that when viewing $\pP$ and $\pQ$ as categories, the monotonicity condition implies that the set-theoretic map of posets $f$ defines a functor.
This is because whenever two elements in $\pP$ have a morphism between them, it is guaranteed to be sent to a morphism in $\pQ$.
Being able to pass between order-theoretic and category-theoretic notions will be a very useful skill for reading this paper, so the reader should be on alert for these equivalent ways of saying somethings.

\begin{defn}[Pullback]\label{defn:pullback}
Given any map of posets $f\colon\pP \to \pQ$ and a $\pQ$-module $M$, one obtains a $\pP$-module $f^*M$ called the \define{pullback of $M$ along $f$}, which is defined pointwise by the equation $f^*M(p)=M(f(p))$.

This defines a $\pP$-module because if $f$ is a map of posets, then whenever $p\leq_{\pP} q$ then $f(p) \leq_{\pQ} f(q)$ and we can ``borrow'' the internal map of $M$ there, i.e.~$M(f(p)) \to M(f(q))$ specifies the internal map $f^*M(p) \to f^*M(q)$.

Said using category theory, the pullback is simply the composition of functors that fits into the following diagram:
\[
\xymatrix{\pP \ar[r]^f \ar[rd]_{f^*M} & \pQ \ar[d]^M \\
			& \cat .}
\]
We note the following special case: when $f$ is an injection, written $f\colon\pP \hookrightarrow \pQ$, we say that $f^*M$ is the \define{restriction} of $M$ along $f$.
\end{defn}

For an example of pullbacks in action, we continue~\cref{ex:modules-over-a-simplicial-complex}.

\begin{ex}\label{ex:subdivision-pullback}
  Suppose $K$ is an abstract simplicial complex and $|K|$ is its geometric realization.
  A \define{refinement} of $(|K|,K)$ is another simplicial complex $(|K'|,K')$ with a homeomorphism $\varphi\colon |K'| \to |K|$ where every simplex in $|\sigma|\subseteq |K|$ is the union of images of simplices in $|K'|$, i.e.~$|\sigma|=\cup \varphi(|\sigma_i'|)$.
  This property along with continuity of $\varphi$ implies that $\varphi$ defines a surjective map of posets $\varphi\colon K' \to K$, where each abstract simplex in $K'$ is sent to the simplex in $K$ that it refines.

  Now suppose $F\colon\Cell(K) \to \Vect$ is a cellular cosheaf.
  By precomposing with the (dualized) map of posets $\varphi\colon (K')^{\op}\to K^{\op}=:\Cell(K)$, we obtain a new cellular cosheaf $F'=\varphi^*F\colon\Cell(K')\to\Vect$ defined over $K'$.
\end{ex}

Now we consider the dual problem.
If $f\colon\pP \to \pQ$ is a map of posets and $M$ is a $\pP$-module then we can \emph{pushforward} $M$ to produce a $\pQ$-module $f_*M$.

\begin{defn}[Pushforward]\label{defn:pushforward}
Let $f\colon \pP \to \pQ$ be a map of posets.
Given a $\pP$-module $M$, one can define the \define{pushforward of $M$ along $f$}, written $f_*M$, via the pointwise formula for the \emph{left Kan extension} of $M$ along $f$. See \cref{defn:pointwise-left-Kan} for more detail.
\[
f_* M := \Lan_f M \qquad \text{where} \qquad f_{*} M (q) := \varinjlim_{p \mid f(p)\leq q} M(p).
\]
\end{defn}

One thing to note is that when evaluating the colimit $f_{*} M (q)$ used in \cref{defn:pushforward} the module $M$ is being summarized along a down set.
In particular if $f(p)\leq q$ and $p'\leq p$, then by the definition of a poset map $f(p')\leq q$ as well.

\begin{ex}\label{ex:kan-extend-integers}
  Suppose we have the integers as a subposet of the real numbers, i.e.
  \[
    j\colon \Z \hookrightarrow \R \qquad \text{where} \qquad j(n)=n
  \]
  We can use~\cref{defn:pushforward} to extend any $\Z$-module $M$ to an $\R$-module $j_*M$.
  Following the formula given there we see that
  \[
    j_*M (r) := \varinjlim_{n \mid n \leq r} M(n) = M(\lfloor r \rfloor).
  \]
  In other words, the formula pushforward does what's expected: it rounds down to the largest integer less than $r$ and uses the value of $M$ there.
\end{ex}

\cref{ex:kan-extend-integers} gives rise to a natural question: ``Is there an operation that rounds up, rather than rounds down?''
The answer is yes, and it is handled by another operation.

\begin{defn}\label{defn:pushforward-open-supports}
  Let $f\colon \pP \to \pQ$ be a map of posets.
  Given a $\pP$-module $M$, one can define the \define{pushforward with open supports} of $M$ along $f$, written $f_{\dagger}M$, via the pointwise formula for the \emph{right Kan extension} of $M$ along $f$.
  \[
  f_{\dagger} M := \Ran_f M \qquad \text{where} \qquad f_{\dagger} M (q) := \varprojlim_{p \mid f(p)\geq q} M(p).
  \]
\end{defn}

One of the major themes of this paper is the iteration of a pullback and then a pushforward (of both types) of a module to ``approximate'' or ``pixelate'' a module.
The term ``pixelization'' first appeared in~\cite{Chazal2009b}, but it was not defined using the general notions of pullback and pushforward given here.

\begin{defn}[Pixelization of $M$]\label{defn:pixelization-module}
  Suppose $f\colon\pP \to \pQ$ is a map of posets and $M$ is a $\pQ$-module.
  The \define{(lower) pixelization of $M$ by $f$} is the $\pQ$-module obtained by pulling back and then pushing forward along $f$, written $f_*f^*M$.
  Dually, the \define{upper pixelization of $M$ by $f$} is the $\pQ$-module obtained by pulling back and then applying the pushforward with open supports $f_{\dagger}f^* M$.
  If the ``lower'' or ``upper'' modifier is not clearly stated, the lower pixelization should be assumed, unless it is clear from context.
\end{defn}

\begin{rmk}[Which Pixelization?]\label{rmk:which_pixelization}
  At this point we are faced with two, apparently equally good, choices for the pixelization of a $\pQ$-module along $f$.
  We'll bias our presentation at the beginning towards the lower pixelization.
  In a sense the lower pixelization is the choice that is most compatible with thinking of modules over posets as cosheaves on the down set topology.
  Additionally, the lower pixelization interacts best with standard choices for defining a weak interleaving of modules.
\end{rmk}

The following example hopefully justifies the word choice behind ``pixelization,'' which is most clearly understood in the case where $f$ is the inclusion of a subposet in $\pQ$.
The pullback along $f$ samples the $\pQ$-module $M$ at points in $\pP$.
The pushforwad then extends the values at these sampled points to regions not in the image of $f$.

\begin{ex}[Sampling of a Persistence Module]\label{ex:pixelized-R-module}
  Consider the sampling of an $\R$-module $M\colon\R \to \Vect$ at points regularly spaced at a distance $\delta$ apart.
  One way to do this is to consider the map
  \[
    d\colon \Z \hookrightarrow \R \qquad \text{where} \qquad d(n)=\delta n.
  \]
  Combining the formulas in \cref{defn:pullback} and \cref{defn:pushforward} gives
  \[
    d_*d^*M (r) := \varinjlim_{n \mid \delta n \leq r} M(\delta n) \cong M(\delta \lfloor \frac{r}{\delta} \rfloor)
  \]
  The above formula says that $d_*d^* M$ is the $\R$-module that is constant on each half-open interval of the form $[\delta n, \delta (n+1))$ with the value $M(\delta n)$, where $n=\lfloor \frac{r}{\delta} \rfloor$.
  The ``pixels'' of this pixelization are the half-open intervals $[\delta n,\delta (n+1))$.
  Dually, the ``pixels'' of the upper pixelization go the other way and are of the form $(\delta n,\delta (n+1)]$.
\end{ex}

\begin{figure}[h]
  \includegraphics[width = .45\textwidth]{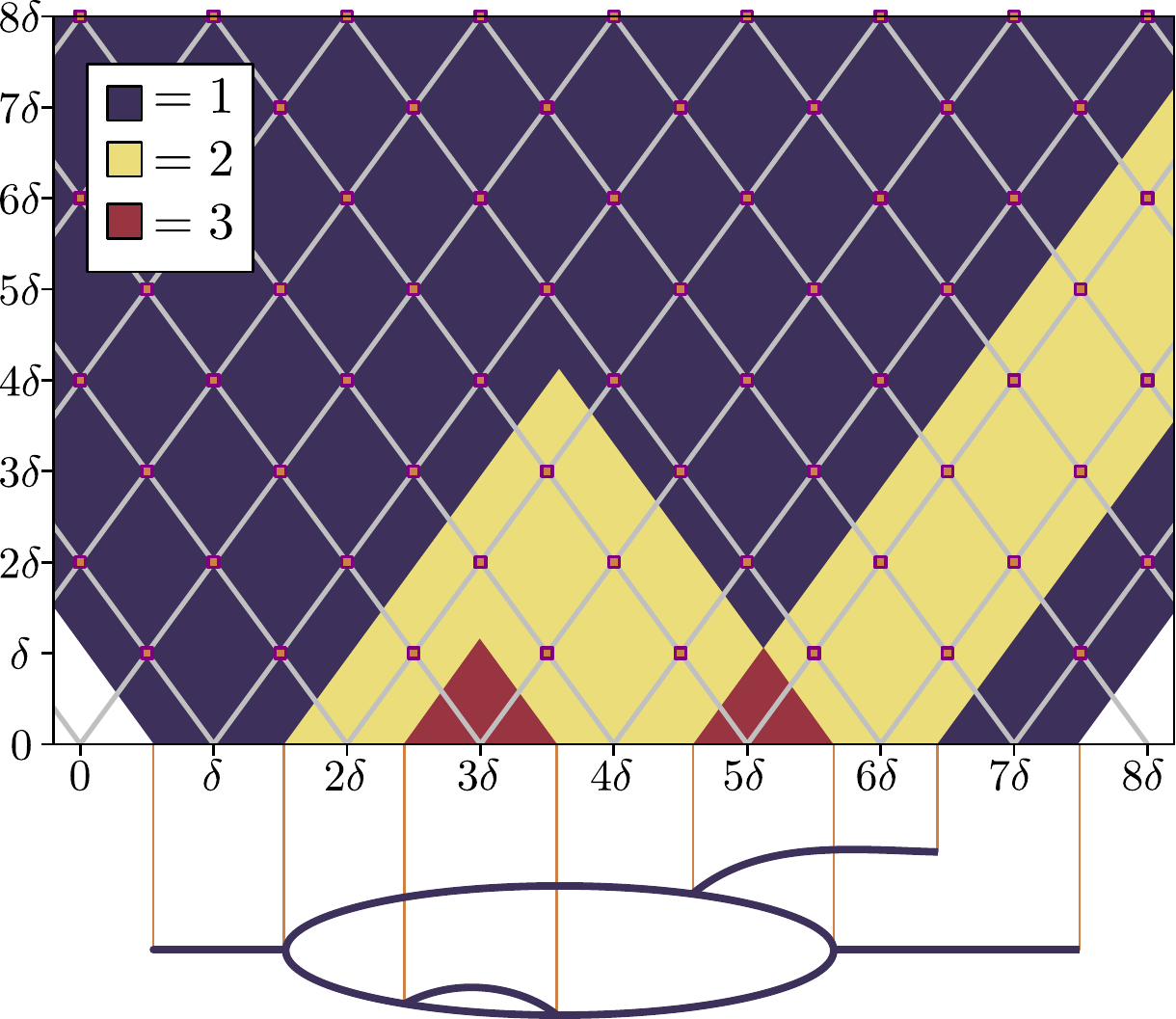}
  \includegraphics[width = .45\textwidth]{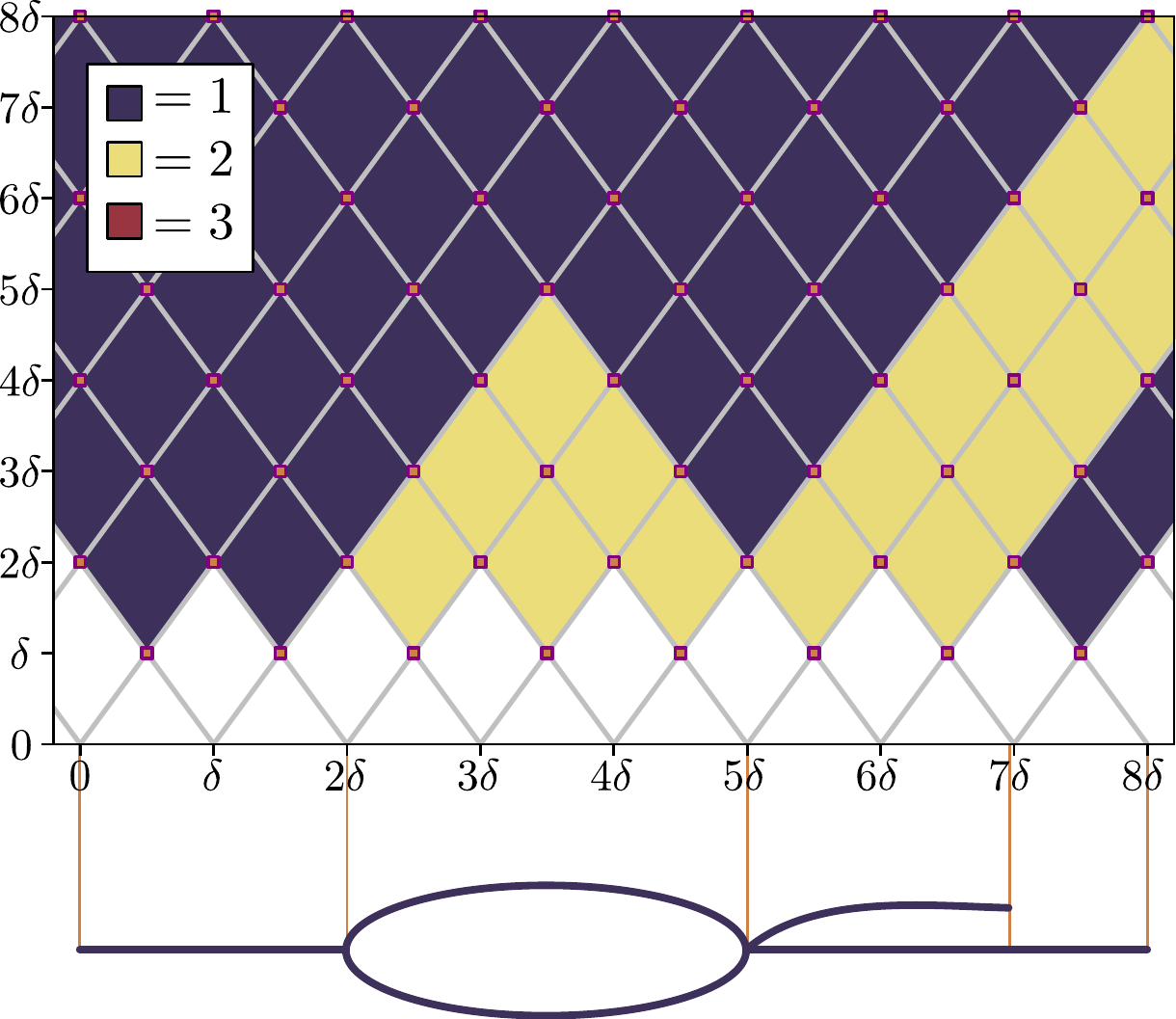}
  \caption{
  An example of the relationship between a Reeb cosheaf $M$, and the lower pixelization.
  For the Reeb graph shown below, color in the left figure shows the size of the set $M(U)$ for a given open interval $U$ drawn in the center-diameter plane that parameterizes open intervals in $\R$. White denotes the empty set.
  The pullback $f^*M$ to $\GridO$ can be read off of the color below the square grid elements.
  The pushforward of the pullback, $f_*f^*M$, is shown in the right figure.
  Note that the the resulting construction is only a precosheaf and not a cosheaf. By ``slicing'' at $2\delta$ we can read off the pixelated Reeb graph drawn below the right figure.
  }
  \label{fig:ReebToMapperExample}
\end{figure}

We now consider an example pertinent to the study of Reeb graphs.

\begin{ex}[Sampling a Reeb Graph]\label{ex:Reeb-graph}
  Let $\pQ=\Int(\R)$ be the poset of open intervals in $\R$.
  Let $\pP=\GridO$ be the subposet of open intervals of the form $(k_1\cdot \delta,k_2\cdot \delta)$ for $k_1<k_2$ integers.
  Given a space $\topX$ and a map $\pi\colon \topX\to\R$, one gets an $\Int(\R)$-module by assigning to every interval $(a,b)$ the set of path components $\pi_0(\pi^{-1}\left((a,b)\right))$ of the pre-image. This is the \define{Reeb cosheaf} associated to an $\R$-space, as defined in~\cite{deSilva2016}.

  In \cref{fig:ReebToMapperExample} we consider a graph along with a projection map to $\R$, which is an example of a Reeb graph.
  We can visualize the Reeb cosheaf $M$ of this Reeb graph by taking each point in the upper half plane to specify the midpoint and diameter of each open interval.
  We use color to indicate the cardinality of the set of path components of the pre-image of each interval in $\R$.
  If $f\colon \GridO \hookrightarrow \Int(\R)$ is the inclusion, then we can read off the lower pixelization of $M$ from the diagram to the right.
  Although this pixelization is not a cosheaf and thus does not determine a Reeb graph a priori, if one ``slices'' the center-diameter plane at $2\delta$ and makes this our new $x$-axis, is does determine a unique cosheaf and hence a new Reeb graph.
\end{ex}

A natural question to ask is ``What is the relationship between the $\R$-module $M$ and its pixelization?''
The functorial answer is that there is a natural map \emph{from} the lower pixelization to the original module.
To see why this is the case, consider
\cref{ex:pixelized-R-module}.
If $r$ is some value between $[\delta n, \delta (n+1))$, then since $\delta n$ is less than $r$, the internal morphism $M(\delta n)\to M(r)$ provides a natural choice of morphism from $d_*d^*M(r) \to M(r)$.
In fact these morphisms collate to form a natural transformation of functors $d_*d^*M \to M$, which is the same thing as a map of $\R$-modules.
Dually, there is a map from the original module \emph{to} its upper pixelization
The following lemma generalizes this observation and is completely standard.

\begin{lem}\label{lem:push-pull-unit-counit}
  Let $f\colon\pP \to \pQ$ be a map of posets and let $f^*\colon\Fun(\pQ,\cat) \to \Fun(\pP,\cat)$ be the pullback and $f_*\colon\Fun(\pP,\cat) \to \Fun(\pQ,\cat)$ be the pushforward functors defined in \cref{defn:pullback} and \cref{defn:pushforward}.
  The pair of functors $(f_*,f^*)$ form an adjoint pair, specifically $f_*$ is \define{left adjoint} to $f^*$.

  The statement that $(f_*,f^*)$ form an adjoint pair is really the statement that
  there are natural transformations
  \[
    \upsilon \colon \id_{\cat^{\pP}} \Rightarrow f^*f_* \qquad \text{and} \qquad \chi \colon f_*f^* \Rightarrow \id_{\cat^{\pQ}}
  \]
  called the \define{unit} and the \define{co-unit} of the adjunction and that
  the unit and counit further satisfy the following \define{triangle identities}, which is explained in further detail in \cref{rmk:triangle-identities}.
  \[
    \id_{f_*} = \chi f_* \circ f_*\upsilon \qquad \text{ and } \qquad \id_{f^*} = f^*\chi \circ \upsilon f^*.
  \]

  Dually, we have that $(f^*,f_{\dagger})$ form an adjoint pair, where the adjunction goes in the opposite direction.
  We reserve the terms $\chi$ and $\upsilon$ for the co-unit and unit of the adjunction, but now decorate these letters with the $\dagger$ symbol to indicate that $f_{\dagger}$ is involved.
  \[
    \chi^{\dagger} \colon f^*f_{\dagger} \Rightarrow \id_{\cat^{\pP}} \qquad \text{and} \qquad \upsilon^{\dagger}\colon\id_{\cat^{\pQ}} \Rightarrow f_{\dagger}f^*
  \]
\end{lem}
\begin{proof}
  This is Proposition 6.1.5 of~\cite{riehl2017category}, but one can refer to \cref{app:lem:push-pull-unit-counit} in the Appendix for a plausibility argument.
\end{proof}

The triangle identities are going to be used in the proof of \cref{lem:full-restricts-interleavings}, which is half of our main result, \cref{thm:extend-restrict-interleavings}.

\begin{rmk}[The Triangle Identities]\label{rmk:triangle-identities}
  Interpreting the triangle identities above can be somewhat tedious, so we spell out exactly their meaning here.

  The equation $\id_{f_*} = \chi f_* \circ f_*\upsilon$ means that for every $\pP$-module $M$ we have that pushing forward the unit map $\upsilon_M\colon M \to f^*f_*M$ along $f$, i.e. the map $f_*\upsilon_M\colon f_*M \to f_*f^*f_*M$, serves as the unique inverse to the counit map on $f_*M$, i.e. $\chi_{f_*M}\colon f_*f^*f_*M \to f_*M$. In other words we have the following commutative triangle:
  \[
    \begin{tikzcd}
      f_* M \ar[r,"f_*\upsilon_M"] \ar[rd,"\id_{f_*M}"'] & f_*f^*f_* M \ar[d,"\chi_{f_*M}"] \\
            & f_*M
    \end{tikzcd}
  \]

  The equation $\id_{f^*} = f^*\chi \circ \upsilon f^*$ means that for every $\pQ$-module $N$ we have that pulling back the counit map $\chi_N \colon f_*f^* N \to N$ along $f^*$, i.e. the map $f^*\chi_N\colon f^*f_*f^* N \to f^* N$, serves as the unique inverse to the unit map on $f^*N$, i.e. $\upsilon_{f^*N} \colon f^*N \to f^*f_*f^*N$.
  \[
    \begin{tikzcd}
      f^*f_*f^* N \ar[rd,"\id_{f^*f_*f^* N}"] \ar[d,"f^*\chi_N"'] & \\
          f^*N \ar[r,"\upsilon_{f^*N}"'] & f^*f_*f^* N
    \end{tikzcd}
  \]
\end{rmk}

We now note an in important special case of this adjunction when $f$ is a full and faithful map of posets.
Recall that a functor $F\colon\cat \to \dat$ is \define{faithful} if for every pair of objects $c,c'\in\cat$, the map on Hom sets
\[
  F(c,c') \colon \text{Hom}_{\cat}(c,c') \to \text{Hom}_{\dat}(F(c),F(c'))
\]
is injective.
The functor $F$ is \define{full} if the above map is surjective.
Notice that when we view the posets $\pP$ and $\pQ$ as categories, where there is at most one morphism between any pair of objects, then a map of posets $f\colon\pP \to \pQ$ is automatically a faithful functor.
To say that the map $f\colon\pP \to \pQ$ is full \emph{and} faithful is to say that the following if and only if condition holds:
\[
p\leq p' \Leftrightarrow f(p) \leq f(p').
\]
Notice that the full condition also implies that the map $f$ is an injection.

\begin{prop}
  If $f\colon\pP \to \pQ$ is a full and faithful map of posets, then $f$ is an injection, which we write as $f\colon\pP\hookrightarrow\pQ$.
\end{prop}
\begin{proof}
  The condition that $f(p)=f(p')$ in a poset is equivalent to the statement that $f(p)\leq f(p')$ and $f(p')\leq f(p)$.
  The full condition of $f$ then implies that $p\leq p'$ and $p'\leq p$ and hence $p=p'$.
\end{proof}

\begin{rmk}
  Although being full and faithful implies that $f$ is an injection, the converse is not true.
  If one takes a discrete poset on $\{a,b\}$ with no relations to be $\pP$ and then one adds the relation $a\leq b$ in order to define $\pQ$, then the inclusion is not full.
\end{rmk}

We now describe the content of~\cite[p.~239, Cor.~3]{MacLane1978} in the special case of a map of posets.

\begin{lem}\label{cor:push-pull-push-iso}
  If $f\colon\pP \hookrightarrow \pQ$ is a full and faithful map of posets, then the unit
  \[
    \upsilon\colon \id_{\cat^{\pP}} \Rightarrow f^*f_*
  \]
  is a natural isomorphism. Dually, for a full and faithful map of posets, the co-unit of the adjunction involving $f_{\dagger}$
  \[
    \chi^{\dagger}\colon f^*f_{\dagger} \Rightarrow \id_{\cat^{\pP}}
  \]
  is a natural isomorphism.
\end{lem}

\section{A Philosophical Overview of Interleavings}
\label{sec:philosophical-interleavings}

In this section we review the classical notion of an interleaving of $\R$-modules and some of its context\footnote{For broader context, see the introduction in~\cite{bubenik2017interleaving}.} in topological data analysis.
Our discussion of the Induced Matching Theorem is meant to motivate the \emph{shift} of a module, which is the core ingredient of any interleaving theory.
By reconsidering the problematic posets of zig-zag type from the introduction, we are led to three philosophical working hypotheses for how a shift of a generalized persistence module should be defined.
The upshot of this discussion is the logic of why we should define interleavings using the embedding of a poset into its lattice of down sets.
The reader is encouraged to proceed linearly through this section, although technically one can move to \cref{sec:relative-interleavings} directly and refer back to \cref{ssec:extrinsic-intrinsic-shifts} and \cref{ssec:thickening-and-shifts} as necessary.

\subsection{Review of Stability of $\R$-modules}

One of the nice things about classical sub-level set persistent homology, where one studies $\R$-modules, is that one can dispense with functorial language and work with an isomorphic combinatorial object known as the \emph{barcode}.
The barcode exists by virtue of a beautiful classification result coming from representation theory.
This result~\cite{crawley2015decomposition} says that any $\R$-module that is pointwise finite dimensional, written $M\colon \R \to \vect$, is completely described, up to isomorphism, by a multiset $B(M)$ of intervals in $\R$, called the \emph{barcode of $M$}.
\footnote{Note that we use $\vect$ for finite dimensional vector spaces, as opposed to the more general $\Vect$.}
The same statement is true for any totally ordered set~\cite{crawley2015decomposition,botnan2018decomposition}.
Effectively obtained by a clever change of basis, the bars in the barcode pick out ``features'' that are \emph{born} at certain radii and which \emph{die} at later radii.
The \emph{persistence} of one of these features is given by the difference between the death and the birth time, alternatively viewed as the length of the corresponding bar in the barcode.
Typically in applications of topological data analysis, long bars are interpreted as robust, significant topological features of the point cloud $X$ and short bars are viewed as noise.

However, in order for persistent homology to be a useful data science tool, one needs to be able to compare \emph{numerically} persistent homology computations arising from different data sets or different subsamplings of the same data set.
In particular, one would like to know if persistence is stable, e.g.~that point clouds that are nearby in the Hausdorff distance produce nearby persistence modules and hence nearby barcodes.
This is done by working with an algebraic generalization of the notion of Hausdorff distance, called the \emph{interleaving distance}, which is stable \cite{Cohen-Steiner2007,Chazal2009b,Chazal2016}.
The interleaving distance not only answers theoretical questions of stability, but it provides an interesting theoretical concept in algebra that provides a notion of ``approximate'' isomorphism of persistence modules.

To describe the interleaving construction, we restrict ourselves to $\Z$-modules for ease of notation and to make our generalization to arbitrary posets more suggestive.
Keeping in line with the viewpoint that interleavings give a notion of approximate isomorphism, a $0$-interleaving between two modules $M$ and $N$ is simply an isomorphism of persistence modules:
\begin{center}
\begin{tikzpicture}[scale=1.5][baseline= (a).base]
\node[scale=0.8] (a) at (0,0){
 \begin{tikzcd}
\cdots\ar[r] & M(i) \ar[r]\ar[d, color=blue, bend right] &M(i+1)\ar[r]\ar[d, color=blue,bend right] & M(i+2)\ar[r]\ar[d, color=blue, bend right] &\cdots\\
\cdots\ar[r]& N(i)\ar[r]\ar[u, color=red, bend right]&N(i+1)\ar[r]\ar[u, color=red, bend right] &N(i+2)\ar[r]\ar[u, color=red, bend right] &\cdots
\end{tikzcd}};
\end{tikzpicture}
\end{center}
A $1$-interleaving is a collection of slanted morphisms such that the following diagram commutes:
\begin{center}
\begin{tikzpicture}[scale=1.5][baseline= (a).base]
\node[scale=0.8] (a) at (0,0){
 \begin{tikzcd}
\cdots\ar[r]\ar[dr, color=blue] & M(i) \ar[r]\ar[dr, color=blue] &M(i+1)\ar[r]\ar[dr, color=blue] & M(i+2)\ar[r]\ar[dr, color=blue] &\cdots\\
\cdots\ar[r]\ar[ur,color=red]& N(i)\ar[r]\ar[ur, color=red]&N(i+1)\ar[r]\ar[ur, color=red] &N(i+2)\ar[r]\ar[ur, color=red] &\cdots
\end{tikzcd}};
\end{tikzpicture}
\end{center}
A $2$-interleaving is a collection of morphisms $\{M(i) \to N(i+2)\}$ and $\{N(i) \to M(i+2)\}$ making a similar diagram commute.
The reader is well-posed to infer the general definition of an $\epsilon$-interleaving, although it is defined precisely in \cref{defn:R-interleaving} below.

Note that a $1$-interleaving requires that the morphism $M(i) \to M(i+2)$ factor through $N(i+1)$ in a natural way.
Intuitively, this means that every ``2-persistent'' feature of $M(i)$---that is a bar of length at least two that overlaps the index $i$---is also present in $N(i+1)$, and vice versa.
This intuition is made precise in the Induced Matching Theorem of~\cite{Bauer2015a}, which provides an alternative proof of the celebrated Isometry Theorem~\cite{Lesnick2012}.
The Isometry Theorem states that the interleaving distance (defined below) between persistence modules is equal to the \define{bottleneck distance} between the barcodes $B(M)$ and $B(N)$, which is defined in terms of matching bars in the respective barcodes and can be computed in polynomial time using the Hungarian algorithm or other algorithms, e.g.~\cite{kerber2017geometry}.
The paper~\cite{Bauer2015a} shows that a $1$-interleaving guarantees that the left and right endpoints of any matched pair of bars differ by at most 1, and every bar of length at least 2 is matched\footnote{In the computation of the bottleneck distance between $B(M)$ and $B(N)$, intervals in each barcode can go unmatched, but at a penalty proportional to their length.}.

Algebraically, the slanted morphisms used in the definition of an $\epsilon$-interleaving are morphisms between one module and the \emph{shift} of the other module.

\begin{defn}\label{defn:R-shift}
Let $M$ be an $\R$-module and let $\epsilon\in [0,\infty)$ be a non-negative real number.
Consider the map of posets $T_{\epsilon}\colon\R\to \R$ that sends a number $t$ to $t+\epsilon$.
The \define{$\epsilon$-shift of $M$} is the pullback of $M$ along $T_{\epsilon}$, i.e.~$M^{\epsilon}:=T^*_{\epsilon}M$.
Consequently, we have $M^{\epsilon}(t)=M(t+\epsilon)$ and the linear map $M^{\epsilon}(t)\to M^{\epsilon}(s)$ is just the map $M(\epsilon+t)\to M(\epsilon+s)$.

The \define{$\epsilon$-shift functor}, written $(-)^{\epsilon}\colon\Vect^{\R} \to \Vect^{\R}$, sends a module $M$ to its shift $M^{\epsilon}$ and sends a map of modules $\psi\colon M \to N$ to the associated map $\psi^{\epsilon}(t) = \psi(\epsilon+t)\colon  M^{\epsilon}(t) \to N^{\epsilon}(t)$.
\end{defn}

The map $T_{\epsilon}$ in \cref{defn:R-shift} has the two properties of a translation, as introduced in \cref{defn:translation}.
In particular there is a natural transformation from the identity map $\id_{\R}\colon \R \to \R$ to $T_{\epsilon}$.
This observation implies a corresponding natural transformation on the level of modules, i.e.~$\eta^{\epsilon}\colon \id_{\Vect^{\R}} \Rightarrow (-)^{\epsilon}$.
To see this, let $\eta_M^{\epsilon}\colon M \to M^{\epsilon}$ be the morphism whose restriction to each $M(t)$ is the internal morphism $M(t)\to M(t+\epsilon)$.
This natural transformation is used to give a succinct, equational way of defining an $\epsilon$-interleaving.

\begin{defn}\label{defn:R-interleaving}
Given $\epsilon\in [0,\infty)$, an \define{$\epsilon$-interleaving of $\R$-modules} $M,N\colon\R\to \cat$ is a pair of morphisms $\psi\colon  M\to N^\epsilon$ and $\varphi\colon N\to M^\epsilon$ such that
$\varphi^\epsilon\circ \psi = (\eta_M^\epsilon)^\epsilon = \eta_M^{2\epsilon}$ and $\psi^\epsilon\circ \varphi = (\eta_N^\epsilon)^\epsilon = \eta_N^{2\epsilon}.$
\label{def:1DInterleaving}
\end{defn}

\begin{rmk}\label{rmk:twice-shifted}
Due to the nice structure of $\R$, we have that shifting by $\epsilon$ twice is the same as shifting by $2\epsilon$.
This is not always the case, so we are careful to use the notation $((\bullet)^{\epsilon})^{\epsilon}$ to indicate shifting by $\epsilon$ twice.
This observation will be continued later in the paper.
\end{rmk}

\subsection{Interleavings for Zig-Zags and the Necessity of Colimits}
\label{ssec:zig-zag-colimits}

As already indicated in the introduction, topological data analysis motivates the study of $\pP$-modules beyond the simplest totally ordered posets such as $\Z$ or $\R$.
The question of how to defined interleavings for modules indexed by general posets $\pP$ immediately leads to the question of how do we shift a general $\pP$-module by some amount $\epsilon$.
In other words, if $M$ is a $\pP$-module, then how do we defined $M^{\epsilon}$?
Revisiting the simplest instance of a zig-zag poset $\pP$ from the introduction, consider the following $\pP$-module $M$.

\[
M(a) \xleftarrow{\alpha} M(b) \xrightarrow{\beta} M(c).
\]

We want to think of the morphisms $\alpha$ and $\beta$ as being the induced maps on homology given by a zig-zag filtration---they represents the two possible ways of ``stepping through'' the filtration.
If we have a homology class $[\gamma]$ at index $b$ in this zig-zag filtration, then the images $\alpha([\gamma])$ and $\beta([\gamma])$ are possible ``futures'' of the class $[\gamma]$ when stepping through this zig-zag filtration.
Inspired by the earlier discussion for persistence modules and the results of~\cite{Bauer2015a}, we advance the following philosophical working hypotheses for how to define the shift of the zig-zag module $M$ indicated above.

\begin{itemize}
	\item[{\bf H1:}] The 1-shift of a module $M$ evaluated at $b$, written $M^1(b)$, should provide a time-1 peek into the future of $M$.
	The implied notion of time here will be made precise in \cref{ssec:extrinsic-intrinsic-shifts}.\label{item:H1}
	\item[{\bf H2:}] An element, colloquially called a ``feature,'' of $M(b)$ should be considered 1-persistent if and only if its image under all possible futures within time-$1$ is non-zero.\label{item:H2}
	\item[{\bf H3:}] The collection of 1-persistent features of $M(b)$ should be consistently summarized as the image of the natural map $M(b) \to M^1(b)$.\label{item:H3}
\end{itemize}

Proceeding by trial and error, we can see that taking $M^1(b)$ to be the direct sum $M(a)\oplus M(c)$ is ``too big'' by considering the following two example zigzag modules:
\begin{equation*}
\text{Example A:} \qquad
\field \xleftarrow{1} \field \xrightarrow{1} \field
\qquad \qquad \text{Example B:} \qquad
0 \xleftarrow{0} \field \xrightarrow{1} \field
\end{equation*}
For the module in Example A, the left-most vector space $\field = M(a)$ is identified with the right-most $\field = M(c)$ through the middle vector space; the direct sum $\field\oplus \field$ would then be counting the same feature twice.
For the module in Example B, the direct sum $\field\oplus 0$ will not be consistent with the idea that an element should be 1-persistent if and only if its image is non-zero under both of the morphisms.
Thus, we need to define the shift $M^1(b)$ as a quotient of $M(a)\oplus M(c)$ which avoids double-counting and identifies with $0$ appropriately.

The correct categorical construction that captures this intuition is the \define{pushout} $\text{PO}$ of the diagram, which is a special case of the \define{colimit} of the module.
$$\text{PO} \cong M(a)\oplus M(c)/\{(\alpha(x), \beta(x))\mid x\in M(b)\}$$
Below, we have three examples modules and their corresponding pushouts/colimits.
\begin{center}
\begin{tikzpicture}[scale=1.5][baseline= (a).base]
\node[scale=0.8] (a) at (0,0){
 \begin{tikzcd}
 & \text{PO} \cong \field &  &  & \text{PO} \cong 0 &  & & \text{PO} \cong \field^2 &\\
\field\ar[ur, "1", dashed] & & \field\ar[ul, swap, "1", dashed] & 0\ar[ur, "0", dashed] & & \field\ar[ul, "0", swap, dashed] & \field\ar[ur, "1", dashed] & & \field\ar[ul, "1",swap, dashed] \\
& \field\ar[ul, "1"]\ar[ur, swap, "1"]\ar[uu, "1", dashed] & & & \field\ar[ul, "0"]\ar[ur, swap, "1"]\ar[uu, "0", dashed] & & & \field\ar[ur, swap,"0"]\ar[ul, "0"]\ar[uu, "0", dashed] &
\end{tikzcd}};
\end{tikzpicture}
\end{center}
The first two modules, reading from left to right, are the modules from Example A and Example B.
The rightmost example, which we'll call Example C, is worth discussing further.
The one and only feature in $M(b)=\field$ dies under both of the maps, so the image of the map $M(b)\to M^1(b)$ is properly 0, but by peeking into the future, one sees two un-related features come into existence.

Perhaps we feel somewhat convinced that $M^1(b)$ is rightly given by the pushout of the module $M$.
However, this leaves open the question of what $M^1(c)$ (or $M^1(a)$) should be.
To guide our choice, we revisit the second fundamental property of a shift operation: there is a natural transformation (a map of modules) from $M$ to $M^1$.
The somewhat counterintuitive consequence of this property is that the value of $M^1(c)$ depends on the whole diagram for $M$, in order to guarantee that the following diagram commutes:
\[
\begin{tikzcd}
 M(a) \ar[d] & \ar[l] M(b)\ar[d]\ar[r]& M(c)\ar[d]\\
 M^1(a) & \ar[l] M^1(b) \ar[r] & M^1(c)
\end{tikzcd}
\]

Revisiting Examples A and B, the above diagram becomes the two diagrams below.
\[
\begin{tikzcd}
\field \ar[d] & \field \ar[d, "1"]\ar[r, "1"]\ar[l] & \field\ar[d] & & 0 \ar[d] & \field\ar[d, "0"]\ar[r, "1"]\ar[l] & \field\ar[d] \\
M^1(a) & \ar[l] \field\ar[r] & M^1(c)& & M^1(a) & \ar[l] 0 \ar[r] & M^1(c).
\end{tikzcd}
\]
Let's focus on $M^1(c)$ for these two diagrams.
For the diagram depicted to the left, which is associated to Example A, one possible choice that would make the diagram commute is $M^1(c)=\field$.
However, for the diagram to the right, which is associated to Example B, we cannot choose $M^1(c)=\field$ because then the right-most square above would not commute.
This is somewhat paradoxical because if we required that $M^1(c)$ be determined ``locally'' by that portion of the diagram only involving $M(c)$, namely $M(b) \to M(c)$, then we would require that $M^1(c)$ be the same for the two diagrams above, which we have just determined cannot be the case.
Indeed, the only choice for the right square above is $M^1(c)=0$, which is precisely the pushout of the diagram $0\leftarrow \field \rightarrow \field$ and the value of $M^1(b)$.
This suggests that for modules of this form we should have $M^1(c)\cong M^1(b)$ and, symmetrically, $M^1(a)\cong M^1(b)$, however these conclusions are, in a sense, consequences of ``boundary effects'' gotten from the fact that the zig-zag only has 3 terms.

So what construction allows us to look at portions of an arbitrary zig-zag module and allows us to summarize the behavior using a colimit?
The perspective adopted in~\cite{botnan2018algebraic} is to embed the underlying zig-zag poset $\pZ$ into $\R^2$, viewed as a map $j\colon\pZ \hookrightarrow \R^2$ and \emph{extend} a $\pZ$-module $M$ to an $\R^2$-module, using the map $j$.
This is exactly the pushforwad of $M$ along $j$, written $j_*M$, as defined in \cref{defn:pushforward}.

To see how the pushforward operation works for embedding a larger zig-zag into a poset like $\R^2$, consider $\pZ$ embedded as the bottom zig-zag (everything in the row labeled $(0)$ and below) in the poset $\pW$ below.
Let's call this embedding $j\colon\pZ \to \pW$, just as before.

\begin{center}
\begin{tikzpicture}[scale=1.5][baseline= (a).base]
\node[scale=0.7] (a) at (0,0){
 \begin{tikzcd}
 & & ~ &   & ~&  & ~ &   &~ &   & ~&   & ~& \\
(2)&\bullet\ar[ur,dashed] &  & \bullet\ar[ur,dashed]\ar[ul,dashed]  &  & \bullet\ar[ur,dashed]\ar[ul,dashed]   &  & \bullet\ar[ur,dashed]\ar[ul,dashed]   &  & \bullet\ar[ur,dashed]\ar[ul,dashed]   &  & \bullet\ar[ur,dashed]\ar[ul,dashed]   &  & \bullet\ar[ul,dashed] \\
 & & \bullet\ar[ur]\ar[ul]  &   & \bullet\ar[ur]\ar[ul] &  & \bullet\ar[ur]\ar[ul] &   &\bullet\ar[ur]\ar[ul]  &   & \bullet\ar[ur]\ar[ul] &   & \bullet\ar[ur]\ar[ul]&  \\
(1)&\bullet\ar[ur] &  & \bullet\ar[ur]\ar[ul]  &  & \bullet\ar[ur]\ar[ul]   &  & \bullet\ar[ur]\ar[ul]   &  & \bullet\ar[ur]\ar[ul]   &  & \bullet\ar[ur]\ar[ul]   &  & \bullet\ar[ul] \\
 & & \bullet\ar[ur]\ar[ul]  &   & \bullet\ar[ur]\ar[ul] &  & \bullet\ar[ur]\ar[ul] &   &\bullet\ar[ur]\ar[ul]  &   & \bullet\ar[ur]\ar[ul] &   & \bullet\ar[ur]\ar[ul]&  \\
 (0)&\bullet\ar[ur] &  & \bullet\ar[ur]\ar[ul]  &  & \bullet\ar[ur]\ar[ul]   &  & \bullet\ar[ur]\ar[ul]   &  & \bullet\ar[ur]\ar[ul]   &  & \bullet\ar[ur]\ar[ul]   &  & \bullet\ar[ul] \\
 & & \bullet\ar[ur, thick]\ar[ul, thick]  &   & \bullet\ar[ur, thick]\ar[ul, thick] &  & \bullet\ar[ur, thick]\ar[ul, thick] &   &\bullet\ar[ur, thick]\ar[ul, thick] &   & \bullet\ar[ur, thick]\ar[ul, thick]&   & \bullet\ar[ur, thick]\ar[ul, thick]&
\end{tikzcd}};
\end{tikzpicture}
\end{center}

Now consider the following $\pZ$-module $M$:
\begin{center}
\begin{tikzpicture}[scale=1.5][baseline= (a).base]
\node[scale=0.8] (a) at (0,0){
 \begin{tikzcd}
0 &  & 0  &  & \field  &  & \field  &  & 0&  & 0  &  & 0 \\
 & 0 \ar[ur]\ar[ul]&   & 0\ar[ur]\ar[ul]&  & \field\ar[ur, "1", swap]\ar[ul, "1"] &   &\field\ar[ur]\ar[ul, "1"]  &   & 0\ar[ur]\ar[ul] &   & 0\ar[ur]\ar[ul]&
\end{tikzcd}};
\end{tikzpicture}
\end{center}
To evaluate the pushforward $j_*M$ along the described embedding of $\pZ$ into the larger poset $\pW$ depicted above one carries out the following procedure: pick a point $p$ in the larger poset $\pW$, consider the principal down-set $D_p$, take the restriction of $M$ to $j^{-1}(D_p)$, take the colimit of the restriction $M|_{j^{-1}(D_p)}$ to obtain $j_*M(p)$.
If we carry out this procedure for all points $p\in\pW$ one obtains the following $\pW$-module:

\begin{center}
\begin{tikzpicture}[scale=1.5][baseline= (a).base]
\node[scale=0.7] (a) at (0,0){
 \begin{tikzcd}
  & & ~  &   & ~ &  & ~ &   &~ &   & ~&   & ~&  \\
(2)
& \field \ar[ur, gray, dashed] &
& \field \ar[ur, gray, dashed]\ar[ul, gray, dashed]  &
& 0      \ar[ur, gray, dashed]\ar[ul, gray, dashed]   &
& 0      \ar[ur, gray, dashed]\ar[ul, gray, dashed]   &
& 0      \ar[ur, gray, dashed]\ar[ul, gray, dashed]   &
& 0      \ar[ur, gray, dashed]\ar[ul, gray, dashed]   &
& 0      \ar[ul, gray, dashed]
\\
& & \field \ar[ur,swap,"1"] \ar[ul,"1"]
& & \field \ar[ur,swap,"1"] \ar[ul,"1"]
& & 0 \ar[ur] \ar[ul,]
& & 0 \ar[ur] \ar[ul,]
& & 0 \ar[ur] \ar[ul,]
& & 0 \ar[ur] \ar[ul,]
&  \\
(1)
& 0 \ar[ur, gray, dashed] &
& \field \ar[ur, swap,"1", gray, dashed] \ar[ul, "1", gray, dashed]  &
& \field \ar[ur, gray, dashed]           \ar[ul, "1", gray, dashed]  &
& 0 \ar[ur, gray, dashed] \ar[ul, gray, dashed]   &
& 0 \ar[ur, gray, dashed] \ar[ul, gray, dashed]   &
& 0 \ar[ur, gray, dashed] \ar[ul, gray, dashed]   &
& 0 \ar[ul, gray, dashed]
\\
& & 0\ar[ur] \ar[ul]
& & \field   \ar[ur, swap,"1"] \ar[ul, "1"]
& & \field   \ar[ur]           \ar[ul, "1"]
& & 0        \ar[ur] \ar[ul]
& & 0        \ar[ur] \ar[ul]
& & 0        \ar[ur] \ar[ul]
& \\
(0)
& 0\ar[ur, gray, dashed] &
& 0\ar[ur, gray, dashed]\ar[ul, gray, dashed]  &
& \field\ar[ur, gray, dashed]\ar[ul, "1", gray, dashed]   &
& \field\ar[ur, gray, dashed]\ar[ul, "1", gray, dashed]   &
& 0\ar[ur, gray, dashed]\ar[ul, gray, dashed]   &
& 0\ar[ur, gray, dashed]\ar[ul, gray, dashed]   &
& 0\ar[ul, gray, dashed]
\\
& & 0      \ar[ur, thick] \ar[ul, thick]
& & 0      \ar[ur, thick] \ar[ul, thick]
& & \field \ar[ur, thick, swap,"1"] \ar[ul, thick, "1"]
& & \field \ar[ur, thick] \ar[ul, thick, "1"]
& & 0	   \ar[ur, thick] \ar[ul, thick]
& & 0	   \ar[ur, thick] \ar[ul, thick]
&
\end{tikzcd}};
\end{tikzpicture}
\end{center}

The perspective taken in~\cite{botnan2018algebraic} is that the $1$-shift of $M$, written $M^1$ is then the restriction of the extension to the portion of the module between $(0)$ and $(1)$, i.e.
\begin{center}
\begin{tikzpicture}[scale=1.5][baseline= (a).base]
\node[scale=0.8] (a) at (0,0){
 \begin{tikzcd}
0&  & \field &  & \field  &  & 0   &  & 0  &  & 0  &  & 0, \\
& 0\ar[ur]\ar[ul]  &   & \field\ar[ur, swap,"1"]\ar[ul, "1"] &  & \field\ar[ur]\ar[ul, "1"] &   &0\ar[ur]\ar[ul] &   & 0\ar[ur]\ar[ul]&   & 0\ar[ur]\ar[ul]&
\end{tikzcd}};
\end{tikzpicture}
\end{center}
the 2-shift $M^2$ is then
\begin{center}
\begin{tikzpicture}[scale=1.5][baseline= (a).base]
\node[scale=0.8] (a) at (0,0){
 \begin{tikzcd}
\field &  & \field &  & \field  &  & 0   &  & 0  &  & 0  &  & 0 \\
& \field \ar[ur,"1", swap]\ar[ul,"1"]  &   & \field\ar[ur, swap,"1"]\ar[ul, "1"] &  & 0\ar[ur]\ar[ul] &   &0\ar[ur]\ar[ul] &   & 0\ar[ur]\ar[ul]&   & 0\ar[ur]\ar[ul]&
\end{tikzcd}};
\end{tikzpicture}
\end{center}
and so on.
The reader may verify that $M^3$ has two non-trivial entries, while $M^4$ is the $0$ module.

Intuitively the reason why the embedding $j\colon\pZ \hookrightarrow \pW$ provides a way of shifting a $\pZ$-module $M$ is that $\pW$ has a natural translation operation.
This translation operation $T$ simply takes a point in row $(i)$ to the corresponding point in row $(i+1)$.
In the next section we are going to consider a more intrinsic way of defining shifts of $\pP$-modules for general $\pP$.

\subsection{Moving from Extrinsic Shifts to Intrinsic Shifts with Time}
\label{ssec:extrinsic-intrinsic-shifts}

At the start of \cref{ssec:zig-zag-colimits} we outlined three philosophical hypotheses for how to shift modules over the simplest possible zig-zag poset, with a view towards what a general theory should be.
Hypotheses \textbf{H1} and \textbf{H2} made informal mention of a notion of ``time,'' whose meaning was deferred until now.
One way in which we can make ``time'' precise is by enriching a poset with a weight function.

\begin{defn}\label{defn:weighted-poset}
A \define{weighted poset} $(\pP, w)$ is a poset $\pP$ with a function $w\colon \pP\times \pP \to \R_{\geq 0}\cup \{\infty\}$ satisfying
\begin{enumerate}
\item $w(p,q) = 0$ for $p\geq q$;
\item $w(p,q) > 0$, whenever $p< q$; and
\item $w(p,q) \leq \omega(p, r) + w(r, q)$ for all $p,q,r\in \pP$.
\end{enumerate}
In other words, a weighted poset is a Lawvere metric~\cite{lawvere1973metric} on the poset category $\pP$ with the additional requirement that $w(p,q) = 0$ whenever $p\geq q$.
\end{defn}

\begin{ex}[Examples A and B continued]\label{ex:weighted_zigzag}
	If we weight the poset used in Examples A and B above, so that each arrow has weight one, then one obtains the weighted poset $(\pZ,w)$:
	\begin{equation*}
	\label{eq:weighted_shift_ex}
	a \xleftarrow{1} b \xrightarrow{1} c.
	\end{equation*}
\end{ex}

The axioms of a weighted poset are arrived at in part by emulating the axioms of a metric space, but where the binary relation $\leq$ breaks the symmetry of how ``distances'' are calculated.
A helpful analogy is given by special relativity, where we have the notion of two points in space-time being time-like separated.

\begin{rmk}[Analogy with Special Relativity]\label{rmk:time}
Two points in space-time are \define{time-like separated} if one can travel from one point to the other without faster-than-light travel.
This induces a poset structure on space-time where $(x,t) \leq (y,t')$ if and only if $(y,t')$ is in the future light-cone of $(x,t)$.
The down-set of $(y,t')$ in this partial order is the set of points in space-time that can causally influence any event at $(y,t')$, i.e.~points in the down-set of $(y,t')$ are in ``in the past.''

If one adopts this language for a general poset $\pP$ the relation $p\leq q$ can be read as $q$ is ``in the future'' of $p$.
The weight $w(p,q)$ then provides a lower bound on the amount of time required before $p$ can reach $q$.
The axioms of a weighted poset then have the following interpretations:
\begin{enumerate}
	\item If $q$ is actually in the past of $p$, then no time needs to pass.
	\item If $q$ is in the future of $p$, then some time must pass.
	\item It is impossible to shortcut the elapsed time between two events by moving through a third event $r$, first.
\end{enumerate}
\end{rmk}

In view of the above remark, we can define for any $p$ in a weighted poset $(\pP,w)$ and any $0\leq \epsilon\in [0,\infty)$, the \define{$\epsilon$-time ball} to be
$$\overrightarrow{B}(p;\epsilon) = \{ r\in \pP\mid w(p,r) \leq \epsilon\};$$
this is the set of points within $\epsilon$-time of $p$.
Notice that for the $0$-time ball we have that $\overrightarrow{B}(p;0)=\{ r\in \pP \mid r\leq p\}$ is exactly the down-set $D_p$, which is the set of events in the past of $p$.
Moreover, any $\epsilon$-time ball $\overrightarrow{B}(p;\epsilon)$ is a down-set because if $w(p,q)\leq \epsilon$, then for any $r\leq q$ the triangle inequality for a weighted poset provides
\[
	w(p,r) \leq w(p,q) + w(q,r) = \epsilon + 0 =\epsilon.
\]

Now we can make precise the intuition described in hypothesis \textbf{H1} from \cref{ssec:zig-zag-colimits}.
We want $M^\epsilon(p)$ to be a vector space that summarizes the features within $\epsilon$-time from $p$, which is precisely the colimit of the restriction of $M$ to the subposet $B(p; \epsilon)$
\[
	M^\epsilon(p) := \varinjlim M|_{B(p; \epsilon)}.
\]
This defines a $\pP$-module because the internal morphisms $M^\epsilon(p) \to M^\epsilon(q)$ are induced by universality of colimits.

\begin{ex}[Examples A and B continued]
Note that if we consider the weighted zigzag in \cref{ex:weighted_zigzag} then the fact that $\overrightarrow{B}(b;1)=\pZ$ implies that $\overrightarrow{B}(a;1)=\overrightarrow{B}(c;1)=\pZ$ as well.
Hence the 1-shifts of
\begin{equation*}
\text{Example A:}
\qquad
\field \xleftarrow{1} \field \xrightarrow{1} \field
\qquad \text{and} \qquad
\text{Example B:}
\qquad
0 \xleftarrow{0} \field \xrightarrow{1} \field
\end{equation*}
are
\begin{equation*}
\text{1-shift of A:}
\qquad
\field \xleftarrow{1} \field \xrightarrow{1} \field
\qquad \text{and} \qquad
\text{1-shift of B:} \qquad
0 \xleftarrow{0} 0 \xrightarrow{0} 0.
\end{equation*}
\end{ex}

\subsection{Thickening Down Sets to Define Shifts of Modules}
\label{ssec:thickening-and-shifts}

We can generalize the above section on shifting modules over a weighted poset by recognizing that all we really want is a way of translating down sets.
This leads us to the notion of a \emph{thickening} structure on a poset, which is essentially a super-linear family of translations on $\Down(\pP)$, but with an added locality assumption.
The motivation for this definition comes from understanding methods for ``shifting'' or ``smoothing'' sheaves (or cosheaves) over a metric space.

If $\topX$ is a topological space, then the collection of open subsets forms a poset $\Open(\topX)$, ordered by containment.
Following the work of~\cite{Curry2014,deSilva2016,Munch2016,kashiwara2017persistent} we can define how to smooth/shift functors modeled on the poset of open sets of a metric space by ``thickening'' the open sets.
This thickening operates in the expected way with
\[
	U^{\epsilon}:=U \cup \left(\bigcup_{x\in U} B(x,\epsilon)\right) \qquad \text{where} \qquad B(x,\epsilon):=\{y\in \topX \mid d(x,y) < \epsilon\},
\]
which in turn defines a map of posets
\[
T(-,-): \quad \Open(\topX) \times \R_{\geq 0} \to \Open(\topX) \qquad (U,\epsilon) \mapsto T(U,\epsilon):=U^{\epsilon}.
\]
Precomposing an $\Open(\topX)$-module by $T(\bullet, \e)$ then defines a shift for that module.
Recognizing that the collection of down-sets $\Down(\pP)$ serves the role of $\Open(\topX)$ in the Alexandrov topology, we can abstract this particular thickening construction away from its metric origins and isolate it as a new structure, thereby unifying the constructions of~\cite{Curry2014,deSilva2016} and~\cite{botnan2018algebraic}.

\begin{defn}\label{defn:thickening}
A \define{thickening} on $\pP$ is a map of posets
$$T(-,-)\colon \Down(\pP)\times \R_{\geq 0} \to \Down(\pP)$$
satisfying the axioms listed below.
Note that for the sake of cleaner notation, we will adopt the convention that $T(-,\epsilon)=:T_{\epsilon}$ and that for $S\in\Down(\pP)$ we have $T_{\epsilon}(S)=:S^{\epsilon}$.

The axioms of a thickening are as follows:
\begin{enumerate}
	\item (Identity) $T_0\colon \Down(\pP) \to \Down(\pP)$ is the identity.
	\item (Subadditivity) For all $\delta, \epsilon \geq 0$ and $S\in \Down(\pP)$ we have the containment
	\[
		(S^{\epsilon})^{\delta} \subseteq S^{\epsilon+\delta}.
	\]
	\item (Locality) For each $\epsilon\geq 0$ and for any down-set $S \in \Down(\pP)$, the union of $\{D_p^{\epsilon}\}_{p\in S}$ is equal to the thickening of $S$, i.e.
	\[
		\bigcup_{p\in S} T_{\epsilon}(D_p)
		= T_{\epsilon}\left(\bigcup_{p\in S} D_p\right)
		=T_{\epsilon}(S)
	\]
\end{enumerate}
\end{defn}

\begin{defn}
\label{def:weakThickening}
A \define{weak thickening} is a thickening where we relax the axiom of the identity and instead only require that there be a natural transformation $\id_{\Down(\pP)} \Rightarrow T(-,0)$.
\end{defn}

We can now define the $\epsilon$-shift of a $\pP$-module for any poset $\pP$ equipped with a thickening as follows.

\begin{defn}\label{defn:shift-module}
Given a $\pP$-module $M\colon\pP \to \cat$ valued in a co-complete category $\cat$ and a (weak) thickening $T$ on $\pP$, we define the \define{$\epsilon$-shift of $M$} in three steps:
\begin{enumerate}
	\item[Step 1:] We take the left Kan extension of $M\colon\pP \to \cat$ along $\iota\colon\pP\to \Down(\pP)$ to obtain the $\Down(\pP)$-module $\iota_*M=:\cosheaf{M}$.
	\item[Step 2:] We pullback $\cosheaf{M}$ along $T_{\epsilon}\colon \Down(\pP) \to \Down(\pP)$ to obtain a module $\cosheaf{M}^{\epsilon}:=\cosheaf{M}\circ T_{\epsilon}$.
	\item[Step 3:] We restrict $\cosheaf{M}^{\epsilon}$ along the inclusion $\iota\colon\pP \to \Down(\pP)$ to obtain $M^{\epsilon}:=\iota^*\cosheaf{M}^{\epsilon}$.
\end{enumerate}
Said using diagrams, we define $M^{\epsilon}$ using the curved top arrow:
\[
\xymatrix{\pP \ar[d]_{\iota}  \ar@/^2pc/[rr]^{M^{\epsilon}} & \pP \ar[r]^M \ar[d]_{\iota} & \cat \\
		\Down(\pP) \ar[r]_{T_{\epsilon}} & \Down(\pP) \ar[ur]_{\cosheaf{M}} & }
\]
Said using push-pull notation, we can also write the shifted module as
\[
  M^{\e}:= \iota^*T_{\e}^*\iota_* M
\]
Point-wise we can define $M^{\epsilon}$ using the formula
\[
M^{\epsilon}(p):=\cosheaf{M}(D_p^{\epsilon})
\]
where $D_p$ is the principal down-set at $p\in\pP$.
This construction is clearly functorial, thereby giving us a \define{shift functor} for $\pP$-modules:
\[
(-)^{\epsilon}\colon \Fun(\pP,\cat) \to \Fun(\pP,\cat).
\]
\end{defn}
\section{A Relative Theory of Interleavings}
\label{sec:relative-interleavings}

In this section we give an abstract treatment of how to relativize the interleaving construction of Bubenik, de Silva, and Scott~\cite{Bubenik2014a}.
From this we mean that we are going to consider how to define interleavings of $\pP$-modules when we are given the data of
\begin{itemize}
  \item a map of posets $f\colon\pP \to \pQ$, and
  \item a super-linear family of translations on $\pQ$.
\end{itemize}
Let's recall this latter definition.

\begin{defn}
  The collection of translations on a poset $\pQ$, denoted $\Trans_{\pQ}$, is a poset, where $T \leq T'$ if and only if for every $q\in\pQ$ we have $T(q)\leq T'(q)$.
\end{defn}

\begin{defn}\label{defn:family-of-translations}
  A \define{family of translations} is a function $T_{\bullet} \colon [0,\infty) \to \Trans_{\pQ}$, which associates to each $\e\in [0,\infty)$ a translation $T_{\e}\colon\pQ \to \pQ$.
  Such a family is \define{superlinear} if $T_{\e_2}\circ T_{\e_1} \leq T_{\e_1 +\e_2}$ for all $\e_1,\e_2\in [0,\infty)$.
  Note that every translation is by definition greater than the identity, i.e.~$\id_{\pQ} \leq T_{\e}$ for every $\e\in[0,\infty)$.
\end{defn}

\begin{rmk}
  The superlinear condition automatically makes $T_{\bullet}$ into a \emph{lax monoidal functor} from $[0,\infty)$ to $\Trans_{\pQ}$, where the monoidal structure on $[0,\infty)$ is addition and the monoidal structure on $\Trans_{\pQ}$ is composition.
\end{rmk}

\begin{defn}\label{defn:shift-over-Q}
  If $\pQ$ is equipped with a superlinear family of translations $T_{\bullet}$, then we have for every $\e$ an $\e$-shift functor
  \[
    (-)_{\pQ}^{\e} \colon \Fun(\pQ,\cat) \to \Fun(\pQ,\cat) \qquad M \,\squigrightarrow \, M^{\e}:=T_{\e}^*M
  \]
  For every $\e\geq 0$ we have a natural transformation from the identity to this functor, i.e.
  \[
    \eta_{\pQ}^{\e} \colon \id_{\cat^{\pQ}} \Rightarrow (-)^{\e}_{\pQ}.
  \]
  This comes from post-composing the natural transformation $\id_{\pQ} \Rightarrow T_{\e}$ with the functor $M\colon\pQ \to \cat$ for any $\pQ$-module $M$.
  Additionally, for every pair $0\leq \e\leq \e'$ we have a natural transformation
  \[
    \eta_{\pQ}^{\e,\e'} \colon (-)^{\e}_{\pQ} \Rightarrow (-)^{\e'}_{\pQ}.
  \]
  Finally, we note that the superlinear condition $T_{\e_1}\circ T_{\e_2} \leq T_{\e_1+\e_2}$ provides another family of natural transformations
  \[
    \Sigma_{\pQ}^{\e_1,\e_2}\colon ((-)_{\pQ}^{\e_1})_{\pQ}^{\e_2} \Rightarrow (-)_{\pQ}^{\e_1+\e_2}
  \]
\end{defn}

We now introduce a variant on the notion of interleaving that was explored in detail by Anastasios Stefanou in his thesis~\cite{stefanou2018} and associated journal article~\cite{de2018theory}.

\begin{defn}\label{defn:weak-interleaving}
  We fix a superlinear family of translations $T_{\bullet}$ on a poset $\pQ$.
  A \define{weak $\e$-interleaving} of two $\pQ$-modules $M$ and $N$ is a pair of morphisms $\varphi\colon M \to T_{\e}^*N$ and $\psi\colon N \to T_{\e}^*M$ making the following diagram commute.
  \[
  \xymatrix{
    (M)_{\pQ}^{2\e} & ((M)_{\pQ}^{\e})_{\pQ}^{\e} \ar[l]_{\Sigma^{\e,\e}} & & ((N)_{\pQ}^{\e})_{\pQ}^{\e} \ar[r]^{\Sigma^{\e,\e}} & (N)_{\pQ}^{2\e} \\
    & (M)_{\pQ}^{\e} \ar[urr]_(.3){\varphi^{\e}} & & (N)_{\pQ}^{\e} \ar[ull]^(.3){\psi^{\e}} & \\
    (M)_{\pQ}^0 \ar[uu]^{\eta^{0,2\e}} & M \ar[l]^{\eta^0} \ar[urr]_(.25){\varphi} & & N \ar[r]_{\eta^0} \ar[ull]^(.25){\psi} & (N)_{\pQ}^0 \ar[uu]_{\eta^{0,2\e}}
  }
  \]
  We say $M$ and $N$ are \define{weakly $\e$-interleaved} if there exists a weak $\e$-interleaving between them.
  The \define{weak interleaving distance} between two $\pQ$-modules $M$ and $N$ is
  \[
    d_{\pQ}(M,N):=\inf \{\e \mid \exists\, \text{a weak } \e\text{- interleaving}\}
  \]
\end{defn}

\begin{rmk}[Weak versus Standard Interleavings]
  The definition of a weak interleaving is sometimes also called a \define{pentagonal interleaving}, because it requires that a pair of intertwined pentagons commute.
  This is meant to stand in contrast to the usual definition of an \define{$\e$-interleaving}, which requires that a pair of intertwined triangles commute.
  Here we drop the subscript $\pQ$ since we won't be using this notion of interleaving any further.
  \[
  \xymatrix{
    (M^{\e})^{\e} & & (N^{\e})^{\e} \\
     M^{\e} \ar[urr]_(.3){\varphi^{\e}} \ar[u]^{\eta^{\e}} & & N^{\e} \ar[ull]^(.3){\psi^{\e}} \ar[u]_{\eta^{\e}} \\
     M \ar[u]^{\eta^{\e}} \ar[urr]_(.25){\varphi} & & N \ar[u]_{\eta^{\e}} \ar[ull]^(.25){\psi}
  }
  \]
  We note that a standard interleaving always implies a weak interleaving, because if $M$ and $N$ have a standard $\e$-interleaving then the morphism $M\to (M)^{2\e}$ factors through $(M^{\e})^{\e}$.
  On the other hand, if we have a weak $\e$-interleaving, then one can check that there is at least a standard $2\e$-interleaving.
  We summarize these observations by the following string of inequalities.
  \[
    d_{\text{weak}}(M,N) \leq d_{\text{standard}}(M,N) \leq 2 d_{\text{weak}}(M,N)
  \]
\end{rmk}

\subsection{Interleaving Over the Domain of a Poset Map}

In this section we explore two pathways to defining an interleaving over the domain of a poset map $f\colon \pP\to\pQ$, under the assumption that one knows how to interleave modules over $\pQ$.
One path is to simply Kan extend a $\pP$-module to a $\pQ$-module and use the interleaving diagram in \cref{defn:weak-interleaving}, while working exclusively over $\pQ$.
We show that this path has an intrinsic counterpart defined using modules over $\pP$ and colimits of portions of these modules.
In later sections this allows us to enjoy the theoretical benefits of working in the continuous realm, while practically all computations are conducted in a discretized setting.
Our main result is that these two paths give the same distance between $\pP$-modules.

Now we introduce the key construction and definition of this section.
If we fix a map of posets $f\colon\pP\to\pQ$ and a super-linear family of translations over $\pQ$ we can define the $\e$-shift of a $\pP$-module with respect to $f$ as follows.

\begin{defn}\label{defn:shift-relative-to-f}
  Let $f\colon\pP\to\pQ$ be a map of posets. Equip $\pQ$ with a superlinear family of translations $T_{\bullet}\colon\pQ\times [0,\infty) \to \pQ$.
  For each $\e\in [0,\infty)$ we can define the \define{$\e$-shift of a $\pP$-module $M$ relative to $f$} as
  \[
    (M)_{\pP}^{\e}:=f^*T^*_{\e}f_*M = \Lan_f(M)\circ T_{\e}\circ f.
  \]
  The last expression above is simply rewriting the pullback notation as pre-composition and the pushforward as the left Kan extension.
  Note that this definition is functorial, so we get a functor
  \[
    (-)_{\pP}^{\e} \colon \Fun(\pP,\cat) \to \Fun(\pP,\cat)
  \]
\end{defn}

\begin{ex}[Shifting Over a Point]\label{ex:shift-over-point}
  Let $\pQ$ be the poset with a single element and let $T_{\bullet}$ be the family of identity translations.
  Let $f:\pP \to \pQ$ be the constant map.
  For any $\epsilon\geq 0$, the $\e$-shift of a $\pP$-module $M$ is the constant module with value the colimit of $M$.
  This also gives an example where the zero shift of a module is not the same thing as the original module.
\end{ex}

\begin{rmk}[One of Two Possible Shifts]
  If the reader recalls \cref{defn:pushforward-open-supports}, they'll perhaps note that the above shift is just one of two possible choices.
  This is correct.
  One could define the \define{upper $\e$-shift relative to $f$} as
  \[
    (M)^{\e}_{\dagger} := f^*T^*_{\e}f_{\dagger}M = \Ran_f(M)\circ T_{\e}\circ f
  \]
  This will lead to a different interleaving theory as discussed at the end of the paper.
\end{rmk}

We now check that we have all the necessary natural transformations to define weak interleavings over $\pP$ using the definition of an $\e$-shift that requires that we first pushforward a $\pP$-module to a $\pQ$-module, shift using $(-)_{\pQ}^{\e}$ and then pullback to $\pP$.

\begin{prop}\label{prop:shift-over-P-structure-maps}
  Following the setup in \cref{defn:shift-relative-to-f} we let $(-)_{\pP}^{\e}$ be the $\e$-shift functor relative to $f\colon\pP\to\pQ$.
  For every $\e\geq 0$ we have a natural transformation from the identity to this functor, i.e.
  \[
    \eta_{\pP}^{\e} \colon \id_{\cat^{\pP}} \Rightarrow (-)^{\e}_{\pP} \qquad \text{where} \qquad \eta_{\pP}^{\e} :=\upsilon \circ f^*\eta_{\pQ}f_*.
  \]
  Additionally, for every pair $0\leq \e\leq \e'$ we have a natural transformation
  \[
    \eta_{\pP}^{\e,\e'} \colon (-)^{\e}_{\pP} \Rightarrow (-)^{\e'}_{\pP} \qquad \text{where} \qquad \eta_{\pP}^{\e,\e'} := f^* \eta_{\pQ}^{\e,\e'}f_*.
  \]
  Finally, we note that for every $\e_1,\e_2\geq 0$ we have a natural transformation
  \[
    \Sigma_{\pP}^{\e_1,\e_2} \colon ((-)_{\pP}^{\e_1})_{\pP}^{\e_2} \Rightarrow (-)_{\pP}^{\e_1+\e_2}.
  \]
  This last natural transformation is defined on each $\pP$-module $M$ via the triangle
  \[
    \begin{tikzcd}
      ((M)_{\pP}^{\e_1})_{\pP}^{\e_2} \arrow[r,"\Sigma_{\pP}^{\e_1,\e_2}"] \arrow[d,"\chi"'] & f^*(f_* M)_{\pQ}^{\e_1+\e_2}=:(M)_{\pP}^{\e_1+\e_2} \\
      f^*((f_* M)_{\pQ}^{\e_1})_{\pQ}^{\e_2} \arrow[ur,"f^*\Sigma_{\pQ}^{\e_1,\e_2} f_*"'] &
    \end{tikzcd}
  \]
\end{prop}
\begin{proof}
  Notice that whenever we apply the natural transformation $\eta_{\pQ}^{\e}\colon\id_{\cat^{\pQ}} \Rightarrow (-)^{\e}_{\pQ}$ to a module of the form $f_*M$, then this provides a natural transformation $f_*M \Rightarrow T_{\e}^*f_*M=(f_* M )_{\pQ}^{\e}$.
  Note that in general, for every morphism used in the definition of this natural transformation, we can apply the pullback functor $f^*$ to that morphism to obtain the natural transformation
  \[
    f^*\eta_{\pQ}f_*\colon f^*f_* \Rightarrow f^*T^*_{\e}f_* =: (-)_{\pP}^{\e}.
  \]
  Precomposing this natural transformation with the unit $\upsilon\colon \id_{\cat^{\pP}} \Rightarrow f^*f_*$ from \cref{lem:push-pull-unit-counit}
  defines the desired natural transformation
  \[
    \eta_{\pP}^{\e} =\upsilon \circ f^*\eta_{\pQ}f_* \colon \id_{\cat^{\pP}} \Rightarrow f^*f_* \Rightarrow f^*T^*_{\e}f_* =: (-)_{\pP}^{\e}.
  \]

  The natural transformation $\eta_{\pP}^{\e,\e'}$ is constructed similarly.
  We restrict the morphisms given by the natural transformation
  \[
    \eta_{\pQ}^{\e,\e'} \colon (-)^{\e}_{\pQ} \Rightarrow (-)^{\e'}_{\pQ}
  \]
  to modules of the form $f_*M$ and then pullback those morphisms via $f^*$ to obtain
  \[
    \eta_{\pP}^{\e,\e'} := f^* \eta_{\pQ}^{\e,\e'}f_*.
  \]

  Finally, our natural transformation $\Sigma_{\pP}^{\e_1,\e_2}\colon ((-)_{\pP}^{\e_1})_{\pP}^{\e_2} \Rightarrow (-)_{\pP}^{\e_1+\e_2}$
  is constructed as follows.
  First we recall an equivalent way of writing a domain functor, namely:
  \[
    ((-)_{\pP}^{\e_1})_{\pP}^{\e_2}=f^*T_{\e_2}^*f_*f^*T_{\e_1}^*f_*
  \]
  Using the counit $\chi\colon f_*f^* \Rightarrow \id_{\cat^{\pQ}}$ from \cref{lem:push-pull-unit-counit} on modules of the form $T_{\e_1}^*f_*M$ and pulling back the provided morphism along $f^*T_{\e_2}^*$ gives a natural transformation
  \[
    ((-)_{\pP}^{\e_1})_{\pP}^{\e_2}=f^*T_{\e_2}^*f_*f^*T_{\e_1}^*f_* \Rightarrow f^*T_{\e_2}^*T_{\e_1}^*f_* = f^*((f_*-)_{\pQ}^{\e_1})_{\pQ}^{\e_2}.
  \]
  Now we can apply the natural transformation $\Sigma_{\pQ}^{\e_1,\e_2}\colon ((-)_{\pQ}^{\e_1})_{\pQ}^{\e_2}\Rightarrow (-)_{\pQ}^{\e_1+\e_2}$ to any module of the form $f_*M$ and pull that morphism back via $f^*$ to obtain the desired natural transformation:
  \[
    \Sigma_{\pP}^{\e_1,\e_2}\colon ((-)_{\pP}^{\e_1})_{\pP}^{\e_2} \Rightarrow f^*((f_*-)_{\pQ}^{\e_1})_{\pQ}^{\e_2} \Rightarrow f^*(f_*-)_{\pQ}^{\e_1+\e_2}=:(-)_{\pP}^{\e_1+\e_2}.
  \]
\end{proof}

We now define interleavings over $\pP$ relative to $f\colon\pP\to\pQ$.

\begin{defn}\label{defn:relative-weak-interleaving}
  As before, we fix a map $f\colon\pP\to\pQ$ and superlinear family of translations $T_{\bullet}$ over $\pQ$.
  A \define{weak $\e$-interleaving relative to $f$} of two $\pP$-modules $M$ and $N$ is a pair of morphisms $\varphi\colon M \to N^{\e}:=f^*T_{\e}^*f_*N$ and
  $\psi\colon N \to M^{\e} := f^*T_{\e}^*f_*M$
  \[
  \xymatrix{
    (M)_{\pP}^{2\e} & ((M)_{\pP}^{\e})_{\pP}^{\e} \ar[l] & & ((N)_{\pP}^{\e})_{\pP}^{\e} \ar[r] & (N)_{\pP}^{2\e} \\
    & (M)_{\pP}^{\e} \ar[urr]_(.3){\varphi^{\e}} & & (N)_{\pP}^{\e} \ar[ull]^(.3){\psi^{\e}} & \\
    (M)_{\pP}^0 \ar[uu] & M \ar[l] \ar[urr]_(.25){\varphi} & & N \ar[r] \ar[ull]^(.25){\psi} & (N)_{\pP}^0 \ar[uu]
  }
  \]
  We say $M$ and $N$ are \define{weakly $\e$-interleaved relative to $f$} if there exists a relative weak $\e$-interleaving between them.
  The \define{weak relative interleaving distance} between two $\pP$-modules $M$ and $N$ is
  \[
    d_{f}(M,N):=\inf \{\e \mid \exists\, \text{a weak } \e\text{- interleaving relative to } f\}
  \]
  With all these specific definitions in place, we will often be loose in our language and say that two $\pP$-modules are \emph{interleaved} when they are weakly interleaved relative to $f$ and write $d_{\pP}(M,N)$ for the weak relative interleaving distance between $M$ and $N$.
\end{defn}

\begin{rmk}[Unraveling Notation]
The notation above supresses a whole sequence of operations that the reader should be aware of.
For example,
\[
  ((M)_{\pP}^{\e})_{\pP}^{\e}:=f^*T_{\e}^*f_*f^*T_{\e}^*f_*M.
\]
Also $(M)_{\pP}^{2\e}$ means $f^*T_{2\e}^*f_*M$.
\end{rmk}

\begin{ex}[Interleaving Over a Point]\label{ex:interleave-over-point}
  Continuing \cref{ex:shift-over-point}, we consider the constant map $f:\pP\to \pQ=\{\star\}$ and the family of identity translations over $\pQ$.
  The reader is encouraged to verify for themselves that two $\pP$-modules $M$ and $N$ are weakly $\e$-interleaved relative to $f$ if and only if their colimits are isomorphic.
\end{ex}

Our goal is to now understand when an interleaving over $\pQ$ determines an interleaving over $\pP$ and vice versa.
To this end, we first record an immediate corollary of \cref{lem:push-pull-unit-counit}.

\begin{cor}\label{cor:shift-push-to-push-shift}
  Suppose we have defined the shift operation over $\pQ$ as in \cref{defn:shift-over-Q} and the shift over $\pP$ as in \cref{defn:shift-relative-to-f}, then we have the following natural transformation of functors, which is also natural in $\e$.
  \[
    f_*(-)^{\e}_{\pP} \Rightarrow (f_* -)^{\e}_{\pQ}.
  \]
\end{cor}
\begin{proof}
  By using the co-unit of the adjunction in \cref{lem:push-pull-unit-counit}
  \[
    \chi \colon f_*f^*\Rightarrow \id_{\cat^{\pQ}}
  \]
  the stated natural transformation comes from
  \[
    f_*(-)_{\pP}^{\e}=f_*f^*T_{\e}^*f_* \Rightarrow T_{\e}^*f_* =: (f_* -)_{\pQ}^{\e}
  \]
\end{proof}

We now have all the necessary ingredients to state and prove our first main lemma of this section.
The statement of \cref{lem:Kan-extend-interleavings} is that with the setup above an $\e$-interleaving over $\pP$ always implies the existence of an $\e$-interleaving over $\pQ$. What makes this result non-obvious is that when interleaving over $\pP$ one has to apply the shift structure $(-)_{\pP}^{\e}$ twice, which involves iterating a pushforward and pullback operation.
Showing that these operations can be detangled to provide an interleaving over $\pQ$ is the content of the next result.

\begin{lem}\label{lem:Kan-extend-interleavings}
  Let $f\colon\pP \to \pQ$ be an arbitrary map of posets.
  If $M$ and $N$ are weakly $\e$-interleaved relative to $f$ then $f_*M$ and $f_*N$ are weakly $\e$-interleaved over $\pQ$.
  Said more succinctly:
  \begin{quote}
    \begin{center}
      \emph{Left Kan extensions preserve weak interleavings.}
    \end{center}
  \end{quote}
\end{lem}
\begin{proof}
  Assume $M$ and $N$ are $\pP$-modules and $\varphi\colon M \to (N)_{\pP}^{\e}$ and $\psi\colon N \to (M)_{\pP}^{\e}$ are our interleaving morphisms.
  By applying the pushforward along $f$ functor, we get
  \[
    f_*\varphi \colon f_*M \to f_*(N)_{\pP}^{\e}.
  \]
  Now, by virtue of \cref{cor:shift-push-to-push-shift}, we have a morphism
  \[
    f_*(N)_{\pP}^{\e} \to (f_*N)_{\pQ}^{\e},
  \]
  which is defined by evaluating the co-unit $\chi \colon f_*f^*\Rightarrow \id$ on $(f_*N)_{\pQ}^{\e}$.
  Let $\hat{\varphi}$ be the composition of these two morphisms
  \[
    \hat{\varphi}\colon f_*M \to f_*(N)_{\pP}^{\e} \to (f_*N)_{\pQ}^{\e}
  \]
  and define $\hat{\psi}$ analogously.
  We claim these form an interleaving pair for $f_*M$ and $f_*N$ over $\pQ$.

  We check the appropriate pentagon that uses $\hat{\varphi}$ commutes, since the corresponding pentagon for $\hat{\psi}$ will commute using the exact same argument with $M$ and $N$ reversed.
  \[
  \begin{tikzcd}
  f_*M \ar[r, "f_*\varphi"] \ar[dr,"\hat{\varphi}"] & f_*(N)_{\pP}^{\e} \ar[r, "f_*\psi^{\e}"]\ar[d,"\chi"]& f_*((M)_{\pP}^{\e})_{\pP}^{\e})\ar[r] \ar[d,"\chi"]& f_*(M)_{\pP}^{2\e} \ar[d,"\chi"]\\
  & (f_*N)_{\pQ}^{\e} \ar[r, "(f_*\psi)^{\e}"]\ar[dr, "\hat{\psi}^{\e}"]& (f_*(M)_{\pP}^{\e})_{\pQ}^{\e} \ar[r]\ar[d,"\chi"]& (f_*M)_{\pQ}^{2\e}\\
  & & ((f_*M)_{\pQ}^{\e})_{\pQ}^{\e} \ar[ur] &
  \end{tikzcd}
  \]
  Note that the second row is obtained from the first row by applying \cref{cor:shift-push-to-push-shift}, so commutativity of the top two rows is immediate.
  All of the triangles commute by definition.
  In particular, the bottom right triangle of the above diagram commutes by \cref{prop:shift-over-P-structure-maps} and using \cref{cor:shift-push-to-push-shift}.
\end{proof}

We now show that a weak interleaving defined over $\pQ$ always restricts to a relative interleaving over $\pP$.
We note that it is unclear if the following lemma holds when weak interleavings are replaced with standard interleavings.

\begin{lem}\label{lem:full-restricts-interleavings}
    Let $f\colon \pP \to \pQ$ be a map of posets.
    If $f_*M$ and $f_*N$ are weakly $\e$-interleaved over $\pQ$ then $M$ and $N$ are weakly $\e$-interleaved relative to $f$.
    Said more succinctly:
    \begin{quote}
      \begin{center}
        \emph{For any map of posets, weak interleavings restrict.}
      \end{center}
    \end{quote}
\end{lem}
\begin{proof}
  We now show that an interleaving of pushforward modules over $\pQ$ provides an intrinsic interleaving the original $\pP$ modules.
  Suppose $\varphi\colon f_*M \to (f_* N)^{\e}_{\pQ}$ and $\psi\colon f_*N \to (f_*M)^{\e}_{\pQ}$ are the defining morphisms in a weak interleaving over $\pQ$.
  As we did above, we start by considering the commuting pentagon used to define an interleaving over $\pQ$ involving $\varphi$ first.
  \[
  \begin{tikzcd}
  f_*M \ar[r, "\varphi"] \ar[dd] & (f_*N)_{\pQ}^{\e} \ar[r, "\psi^{\e}"] & ((f_*M)_{\pQ}^{\e})_{\pQ}^{\e} \ar[dd] \\
  & & \\
  (f_*M)^0_{\pQ} \ar[rr] & & (f_*M)^{2\e}_{\pQ}
  \end{tikzcd}
  \]
  Applying $f^*$ we get
  \[
  \begin{tikzcd}
  f^*f_*M \ar[r, "f^*\varphi"] \ar[dd] & f^*(f_*N)_{\pQ}^{\e} \ar[r, "f^*\psi^{\e}"] & f^*((f_*M)_{\pQ}^{\e})_{\pQ}^{\e} \ar[dd] \\
  & & \\
  f^*(f_*M)^0_{\pQ} \ar[rr] & & f^*(f_*M)^{2\e}_{\pQ}
  \end{tikzcd}
  \]
  Now we make two observations.
  First, we use the unit from \cref{lem:push-pull-unit-counit} to observe that there is a natural map from $M$ to $f^*f_*M$.
  Second, we note that $(M)^{\e'}_{\pP}$ is by definition $f^*(f_*M)^{\e'}_{\pQ}$ for any $\e'\geq 0$.
  Consequently we augment and rewrite the above diagram as
  \[
  \begin{tikzcd}
  M \ar[r,"\upsilon"] & f^*f_*M \ar[r, "f^*\varphi"] \ar[dd] & (N)_{\pP}^{\e} \ar[r, "f^*\psi^{\e}"] & f^*((f_*M)_{\pQ}^{\e})_{\pQ}^{\e} \ar[dd] \\
  & & & \\
  & (M)^0_{\pP} \ar[rr] & & (M)^{2\e}_{\pP}
  \end{tikzcd}
  \]
  The claim is that $\hat{\varphi}:=(f^*\varphi)\circ \upsilon_M$ and $\hat{\psi}:=(f^*\psi)\circ \upsilon_N$ are the defining morphisms in an interleaving of $M$ and $N$ over $\pP$.
  The proof of this fact is nearly complete, except we have not incorporated the morphism $\Sigma_{\pP}^{\e,\e}\colon ((M)^{\e}_{\pP})^{\e}_{\pP} \to (M)^{2\e}_{\pP}$.
  Fortunately, we already noted above that this morphism is defined by the composition
  \[
    ((M)_{\pP}^{\e})_{\pP}^{\e} \to f^*((f_*M)^{\e}_{\pQ})^{\e}_{\pQ} \to (M)_{\pP}^{2\e}
  \]
  This means that we can add a commuting triangle with one ``missing'' dashed arrow to our above diagram to obtain:
  \[
  \begin{tikzcd}
  M \ar[r, "\hat{\varphi}"] \ar[dd] & (N)_{\pP}^{\e} \ar[r, "f^*\psi^{\e}"] \ar[d,dashed] & f^*((f_*M)_{\pQ}^{\e})_{\pQ}^{\e} \ar[dd] \\
  & ((M)^{\e}_{\pP})^{\e}_{\pP} \ar[ur] \ar[dr] & \\
  (M)^0_{\pP} \ar[rr] & & (M)^{2\e}_{\pP}
  \end{tikzcd}
  \]
  This dashed arrow is obtained by applying the ``shift over $\pP$'' functor to $\hat{\psi}:=(f^*\psi)\circ \upsilon_N$, i.e. $\hat{\psi}^{\e}:= f^*T^*_{\e}f_* \hat{\psi}$ is the dashed arrow above that commutes with all the existing arrows above.
  Commutativity of the upper triangle holds requires that
  \[
    (f^*T_{\e}^*\chi_{T_{\e}^*f_*M})\circ f^*T^*_{\e}f_*(f^*\psi \circ \upsilon_N) = f^*T^*_{\e}\psi.
  \]
  To see this we rewrite the left hand side as
  \[
    f^*T_{\e}^*(\chi_{T^*_{\e}f_*M} \circ f_*f^*\psi \circ f_*\upsilon_N) = f^*T_{\e}^*(\psi \circ \chi_{f_*N} \circ f_*\upsilon_N).
  \]
  This rewriting comes from the commutative square
  \[
    \begin{tikzcd}
      f_*N \arrow[r,"\psi"] & T_{\e}^*f_*M \\
      f_*f^*f_*N \arrow[u,"\chi_{f_*N}"'] \arrow[r,"f_*f^*\psi"] & f_*f^*T_{\e}^*f_*M \arrow[u,"\chi_{T^*_{\e}f_*M}"]
    \end{tikzcd}
  \]
  Recalling \cref{rmk:triangle-identities} we have the commutative triangle
  \[
      \begin{tikzcd}
        f_* N \ar[r,"f_*\upsilon_N"] \ar[rd,"\id_{f_*N}"'] & f_*f^*f_* N \ar[d,"\chi_{f_*N}"] \\
              & f_*N
      \end{tikzcd}
  \]
  that permits us to rewrite
  \[
    f^*T_{\e}^*(\psi \circ \chi_{f_*N} \circ f_*\upsilon_N) = f^*T_{\e}^*(\psi \circ \id_{f_*N}) = f^*T_{\e}^*\psi = f^*\psi^{\e},
  \]
  which was wanted.
\end{proof}

We put \cref{lem:Kan-extend-interleavings} and \cref{lem:full-restricts-interleavings} together to state our main theorem of this section.

\begin{thm}\label{thm:extend-restrict-interleavings}
  Let a poset $\pQ$ be equipped with a superlinear family of translations $T_{\bullet}$ as in \cref{defn:family-of-translations}.
  This induces a weak interleaving distance on $\Fun(\pQ,\cat)$ in the sense of \cref{defn:weak-interleaving} and
  any poset map $f\colon \pP\to\pQ$ induces a weak relative interleaving distance on $\Fun(\pP,\cat)$ as in \cref{defn:relative-weak-interleaving}.

  \emph{For any map of posets $f\colon \pP \to \pQ$, the pushforward functor
      \[
        f_* \colon \Fun(\pP,\cat) \to \Fun(\pQ,\cat) \qquad \text{where} \qquad M \squigrightarrow f_*M
      \]
      induces an isometry onto its image, meaning that if $M$ and $N$ are $\pP$-modules, then
      \[
        d_{\pP}(M,N)=d_{\pQ}(f_*M,f_*N).
      \]}
\end{thm}
\begin{proof}
  If $M$ and $N$ are $\e$-interleaved over $\pP$ then \cref{lem:Kan-extend-interleavings} guarantees that $f_*M$ and $f_*N$ are $\e$-interleaved over $\pQ$.
  This means that
  \[
    d_{\pP}(M,N)\geq d_{\pQ}(f_*M,f_*N).
  \]
  To see why, suppose for contradiction that $d_{\pP}(M,N)< d_{\pQ}(f_*M,f_*N)$. This implies there exists an $\e'$ for which $M$ and $N$ are $\e'$-interleaved over $\pP$, but for which $f_*M$ and $f_*N$ are not $\e'$-interleaved over $\pQ$.
  However this would contradict \cref{lem:Kan-extend-interleavings} so this is impossible.

  Similarly for $f\colon\pP\rightarrow\pQ$ \cref{lem:full-restricts-interleavings} implies that for every $\e$-interleaving of $f_*M$ and $f_*N$ there is an $\e$-interleaving of $M$ and $N$.
  This implies that
  \[
    d_{\pQ}(f_*M,f_*N) \geq d_{\pP}(M,N).
  \]
  This proves the stated claim.
\end{proof}

We now note an important special case of this result in the setting where $\pQ=\Down(\pP)$ is the lattice of down sets in $\pP$.
Notice that $\iota\colon\pP \hookrightarrow \Down(\pP)$ is a full and faithful map of posets.
If $\pP$ is equipped with a superlinear family of translations $T_{\bullet}$ then we can define a superlinear family of translations of down sets in $\pP$ via the formula:
\[
  \forall S\in \Down(\pP) \qquad \text{let} \qquad T_{\e}(S):=\cup_{p\in S} D_{T_{e}(p)}
\]
Following the above constructions we can shift a $\pP$-module $M$ in two equivalent ways
\[
  \iota^*T_{\e}^* \iota_* M \qquad \text{or} \qquad T_{\e}^*M.
\]
Indeed $(\iota^*T_{\e}^* \iota_* M)(p)= (T_{\e}^*M)(p)$ for all $p\in\pP$.
As such, we obtain the following corollary of \cref{thm:extend-restrict-interleavings}.

\begin{cor}\label{cor:down-set-isometry}
  Let $\pP$ be a poset equipped with a superlinear family of translations $T_{\bullet}$.
  The category of $\pP$-modules $\Fun(\pP,\cat)$ embeds fully, faithfully and isometrically into the category of modules over $\Down(\pP)$, i.e. $\Fun(\Down(\pP),\cat)$.
\end{cor}

\begin{rmk}
  There is more to say here. The image of $\iota_*$ in $\Fun(\Down(\pP),\cat)$ can be identified with the category of \emph{cosheaves} over $\pP$.
  This implies that the interleaving distance for cosheaves over a poset can be computed using pointwise data.
\end{rmk}

\section{Approximation of a Module by its Pixelization}
\label{sec:approximation}

\cref{sec:relative-interleavings} can be viewed as providing a theory for how to define the interleaving distance between modules over $\pP$ when we're given a map $f\colon\pP\to\pQ$ and an interleaving distance between modules over $\pQ$.
Roughly speaking, this theory says that we should take two $\pP$-modules $M$ and $N$, push them forward to be $\pQ$-modules $f_*M$ and $f_*N$, then use shift structures over $\pQ$ and restriction along $f$ in order to define interleavings over $\pP$.
The content of \cref{lem:Kan-extend-interleavings} and \cref{lem:full-restricts-interleavings} is that it doesn't matter if we extend and work totally over $\pQ$ or if we restrict whenever possible to $\pP$, when defining interleavings.

We now consider the opposite end of this relative theory by focusing on modules defined over the codomain of the poset map $f\colon\pP \to \pQ$. Two natural questions emerge:
\begin{enumerate}
  \item[{\bf Q1:}] Can we use the relative interleaving calculation over $\pP$ to infer the interleaving distance over $\pQ$? This would be especially helpful when the poset $\pQ$ is uncountable, such as the collection of open intervals in $\R$, and when $\pP$ is discrete.
  \item[{\bf Q2:}] How much distortion occurs when we pullback and pushforward $M$? In other words, what hypotheses can we put on $f$ in order to bound the interleaving distance between a module and its pixelization, i.e.~$d_{\pQ}(M,f_*f^*M)$?
\end{enumerate}
We show that the first question can be reduced to the second question.
The second question is most naturally addressed when we assume that $\pP$ has the extra structure of being a lattice.
This will bring us back to considering our pushforward with open supports functor $f_{\dagger}$.

\subsection{Bounding Distortion Using the Triangle Inequality}

Suppose we restrict two $\pQ$-modules $M$ and $N$ along $f$ to obtain an $\e$-interleaving of $f^*M$ and $f^*N$ over $\pP$.
By virtue of \cref{lem:Kan-extend-interleavings} we know that $f_*f^*M$ and $f_*f^*N$ are $\e$-interleaved over $\pQ$, but it's possible that $N$ and $M$ are interleaved for smaller values of $\e$.
By the triangle inequality we have that
\[
d_{\pQ}(M,N) \leq d_{\pQ}(M,f_*f^*M) + d_{\pQ}(f_*f^*M, f_*f^*N) + d_{\pQ}(f_*f^*N,N).
\]
However, \cref{thm:extend-restrict-interleavings} provides us with the identity
\[
  d_{\pQ}(f_*f^*M,f_*f^*N)=d_{\pP}(f^*M,f^*N),
\]
which allows us to rephrase Question 1 in terms of Question 2.

\begin{lem}\label{lem:interleaving-distortion}
  If $f\colon\pP\rightarrow \pQ$ is a map of posets then following the set up in \cref{thm:extend-restrict-interleavings} we can conclude that
  \[
    |d_{\pQ}(M,N)-d_{\pP}(f^*M,f^*N)|\leq d_{\pQ}(M,f_*f^*M) + d_{\pQ}(N, f_*f^*N).
  \]
\end{lem}
\begin{proof}
  First we note that the identity $d_{\pQ}(f_*f^*M,f_*f^*N)=d_{\pP}(f^*M,f^*N)$ follows from
  \cref{thm:extend-restrict-interleavings} by setting $M'=f^*M$ and $N'=f^*N$ and using $M'$ and $N'$ in the statement of the theorem.
  Now the triangle inequality says
  \[
  d_{\pQ}(M,N) \leq d_{\pQ}(M,f_*f^*M) + d_{\pQ}(f_*f^*M, f_*f^*N) + d_{\pQ}(N, f_*f^*N),
  \]
  which implies, by using the above identity,
  \[
  d_{\pQ}(M,N) - d_{\pP}(f^*M,f^*N) \leq d_{\pQ}(M,f_*f^*M) + d_{\pQ}(N, f_*f^*N).
  \]
  Now we can use the triangle inequality reversing $M$ and $N$ and their pixelizations.
  \[
  d_{\pP}(f^*M,f^*N):= d_{\pQ}(f_*f^*M, f_*f^*N) \leq d_{\pQ}(f_*f^*M,M) + d_{\pQ}(M,N) + d_{\pQ}(N,f_*f^*N)
  \]
  Invoking the symmetry of the interleaving distance implies
  \[
  d_{\pP}(f^*M,f^*N)-d_{\pQ}(M,N) \leq d_{\pQ}(M,f_*f^*M) + d_{\pQ}(N, f_*f^*N),
  \]
  which implies the stated inequality:
  \[
    |d_{\pQ}(M,N)-d_{\pP}(f^*M,f^*N)|\leq d_{\pQ}(M,f_*f^*M) + d_{\pQ}(N, f_*f^*N)
  \]
\end{proof}

\cref{lem:interleaving-distortion} implies that the distortion in the interleaving distance is bounded by the distance between a module and its pixelization.
We now develop a general setup where we can bound this distance.

\subsection{Pulling Back to a Lattice}
\label{sec:lattice-approximation}

One of the philosophical consequences of \cref{cor:down-set-isometry} is that we can always replace the study of interleavings of modules over a poset $\pP$ with interleavings of modules over an associated complete lattice, namely $\Down(\pP)$.
However, to afford us a more general treatment of applications we will not just work with the specific lattice of down sets in a poset, but rather use any complete lattice $\pL$ and a map $f\colon\pL \to \pQ$ that respects certain structures such as meets and joins.
We remind the reader of what this means.

\begin{defn}\label{defn:complete-lattice}
  We fix a poset $\pL$.
  If for every subset $S\subseteq L$ (including the empty subset) there is a least upper bound of $S$, called the \define{join} and written $\bigvee S$ or $\sup S$, then $\pL$ is a \define{complete join semilattice}.
  If the join only exists for finite subsets $S$, then we say $\pL$ is a \define{join semilattice}.
  Said differently, $\pL$ is a complete join semilattice if and only if it has arbitrary colimits.
  A poset is a join semilattice if and only if it has finite colimits.

  Dually, a poset $\pL$ is a \define{complete meet semilattice} if for every subset $S\subseteq \pL$ there is a greatest lower bound of $S$, which is called the \define{meet} and is written $\bigwedge S$ or $\inf S$.
  Similarly, if the meet only exists for finite subsets then we say $\pL$ is a \define{meet semilattice}.
  One can rephrase the existence of meets in terms of categorical limits.

  A \define{complete lattice} is a poset where arbitrary subsets have meets and joins.
\end{defn}

We want to consider poset maps that respect meets and joins in the domain even though the codomain might not be a lattice.

\begin{defn}\label{defn:respects-meets-joins}
  Suppose $f\colon\pL\to\pQ$ is a map of posets.
  Let $\pL$ be a complete join semilattice.
  In this setting we say that $f$ \define{respects joins} if whenever there is a $q$ such that $f(x) \leq q$ for all $x\in S\subseteq \pL$, then $f(\bigvee S) \leq q$.

  Dually, let $\pL$ be a complete meet semilattice.
  In this setting we say that $f$ \define{respects meets} if whenever there is a $q\in \pQ$ such that $f(x)\geq q$ for all $x\in S\subseteq \pL$, then $f(\bigwedge S)\geq q$.
\end{defn}

\begin{rmk}\label{rmk:preserve-colimits}
  The above conditions specialize to well-known conditions in the situation where $\pQ$ is a complete lattice and not just a poset.
  First we note that in general for a map or lattices $f\colon\pL \to \pQ$ we have
  \[
    \text{for } S \subseteq \pL \qquad f(\bigvee S) \geq \bigvee f(S).
  \]
  This follows from the observation that since $x\leq \sup S$ for all $x\in S$ the fact that $f$ is a map of posets implies that $f(x)\leq f(\sup S)=f(\bigvee S)$.
  In other words $f(\bigvee S)$ is an upper bound for $f(S)$, but it need not be the \emph{least} upper bound, which is $\bigvee f(S)$.
  The respects joins condition implies that $f(\bigvee S) \leq \bigvee f(S)$ and hence
  \[
    f \text{ respects joins } \qquad \Leftrightarrow \qquad f(\bigvee S) = \bigvee f(S).
  \]
  In other words, $f$ preserves colimits, i.e. it is a continuous functor between posets viewed as categories.
  This is also used as a \emph{definition} of a complete join semilattice homomorphism.

  Similarly, when $\pQ$ has meets, the respects meets condition reduces to the statement that $f$ preserves meets and is thus a complete meet semilattice homomorphism.
\end{rmk}

The structure of a complete join semilattice $\pL$ and a map $f\colon\pL\to\pQ$ that respects joins allows us to simplify the computation needed to describe the pushforward of an $\pL$-module.
Dually, the respects meets condition allows us to describe the pushforward with open supports functor much more cleanly.
These observations stem from simpler, lattice-theoretic origins, which stem from the philosophy of Galois connections; see Section 1.4.3 of~\cite{fong2019invitation} for a modern treatment.

\begin{defn}\label{prop:lattice-adjoint-joins}
  Any poset map $f\colon\pL \to \pQ$ from a complete join semilattice induces a map backwards $f_{\flat}\colon\pQ \to \pL$ via the assignment:
  \[
    f_{\flat}(q) = \bigvee f^{-1}(D_q) = \sup \{x \mid f(x) \leq q\}
  \]
  Moreover if $f$ respects joins then $f\circ f_{\flat} \leq \id_{\pQ}$.

  Dually, any poset map $f\colon\pL \to \pQ$ from a complete meet semilattice induces a map backwards $f_{\sharp}\colon\pQ \to \pL$ via the assignment:
  \[
    f_{\sharp}(q) = \bigwedge f^{-1}(U_q) = \inf\{x \mid f(x) \geq q\}
  \]
  Moreover if $f$ respects meets then $f\circ f_{\sharp} \geq \id_{\pQ}$.
\end{defn}

The existence of the above pair of maps $f_{\flat}$ and $f_{\sharp}$ allows us to phrase the two flavors of pushforwards in terms of pullbacks along these maps.

\begin{prop}\label{prop:pushforward-join}
  If $f\colon\pL\to\pQ$ respects joins then
  for any $\pL$-module $M$
  \[
    (f_*M)(q) := \varinjlim_{x\mid f(x)\leq q} M(x) \cong M(\bigvee f^{-1}(D_q))=:f_{\flat}^*M(q).
  \]
  Dually, if $f\colon\pL \to \pQ$ respects meets then for any $\pL$-module $M$
  \[
    (f_{\dagger}M)(q):= \varprojlim_{x \mid q \leq f(x)} M(x) \cong M(\bigwedge f^{-1}(U_q))=:f_{\sharp}^*M(q).
  \]
\end{prop}
\begin{proof}
  Recall that the $D_q$ above is the principal down set at $q$ so that $f^{-1}(D_q)=\{x \mid f(x) \leq q \}$.
  Since $\pL$ is a \emph{complete} lattice we can conclude that $f^{-1}(D_q)$ has a supremum, denoted $x^{\ast}:=\bigvee f^{-1}(D_q)$.
  This means that whenever $f(x)\leq q$ then we immediately know that $x\leq x^{\ast}$.
  The \emph{respects joins} condition implies that $f(x^{\ast})\leq q$ as well, so $x^{\ast}$ is in the comma category $(f\downarrow q)$.
  We now verify that the inclusion of the supremum
  \[
    i_q\colon x^{\ast}=\bigvee f^{-1}(D_q) \hookrightarrow (f\downarrow q)
  \]
  is cofinal.
  As outlined in \cref{defn:cofinal}, this requires checking non-emptiness and connectedness.
  First it is obvious that for every $x\in (f\downarrow q)$ the comma category $(x\downarrow i_q)$ is non-empty. This follows because if $x\in (f\downarrow q)$ then $f(x) \leq q$ and hence $x\leq x^{\ast}$.
  Connectedness is again immediate because the domain of $i_q$ is a one object category with a single identity morphism.
  This proves the first claim.

  The second claim is similar, but we sketch the basic insights required.
  First we note that $f^{-1}(U_q):=\{x \mid f(x) \geq q\}$.
  Since $\pL$ has meets, there's a greatest lower bound denoted by $x^{\dagger}=\bigwedge f^{-1}(U_q)$.
  Consequently whenever $f(x) \geq q$ then $x^{\dagger} \leq x$.
  The respects meets condition implies that $f(x^{\dagger}) \geq q$ so we have that $x^{\dagger}$ is in the comma category $(f\uparrow q)$.
  By dualizing the proof above one can see that $x^{\dagger}$ is final.
\end{proof}

We now note some important relationships between $f$, $f_{\flat}$ and $f_{\sharp}$ that echo the observations of \cref{lem:push-pull-unit-counit} and \cref{cor:push-pull-push-iso}.

\begin{lem}\label{lem:lattice-adjoints-full}
  If $f\colon\pL \to \pQ$ is a map of posets from a complete join semilattice, then $f_{\flat}\circ f \geq \id_{\pL}$.
  Moreover, if $f$ is full then $f_{\flat} \circ f = \id_{\pL}$.

  Dually, if $f\colon\pL \to \pQ$ is a map of posets from a complete meet semilattice, then $f_{\sharp}\circ f \leq \id_{\pL}$.
  Moreover, if $f$ is full then $f_{\sharp}\circ f = \id_{\pL}$.
\end{lem}
\begin{proof}
  For every $x\in\pL$ we have that
  \[
    f_{\flat}(f(x)) = \sup\{x' \mid f(x') \leq f(x)\}
  \]
  Clearly $x$ is in the set on the right since $f(x)\leq f(x)$, so the supremum is greater than $x$. The fullness assumption guarantees that $f(x')\leq f(x) \Rightarrow x'\leq x$ so $x$ is an upper bound and hence a least upper bound.
  The dual statement for $f_{\sharp}$ follows the exact same line of reasoning.
\end{proof}

The following question immediately comes to mind:
If we have a map $f\colon\pL \to \pQ$ from a complete lattice and a superlinear family of translations on $\pQ$, then by post-composition we get another map $T_{\e}\circ f \colon\pL \to \pQ$.
If we apply $f_{\flat}$ or $f_{\sharp}$ backwards, does $\pL$ gain a superlinear family of translations?
It turns out that this only is true for $f_{\flat}$ and \emph{not} for $f_{\sharp}$.
Attempting to use $f_{\sharp}$ produces a sublinear family of translations, which requires a different interleaving theory.

\begin{lem}\label{lem:lower-approximate-translation}
  Let $f\colon\pL \to \pQ$ be a poset map from a complete join semilattice to a poset $\pQ$ that respects joins, see \cref{defn:respects-meets-joins}.
  If $\pQ$ is equipped with a superlinear family of translations $T_{\bullet}$, then
  \[
    T^{\flat}_{\e}:=f_{\flat} \circ T_{\e} \circ f,
  \]
  which we call the \define{lower approximation translation},
  defines a superlinear family of translations on $\pL$.
  Moreover, this translation obeys
  \[
    f(T^{\flat}_{\e}(x)) \leq T_{\e}(f(x)).
  \]
\end{lem}
\begin{proof}
  First we note that if $T_{\e}$ is a translation then $T_{\e}(f(x))\geq f(x)$.
  Applying $f_{\flat}$ and applying \cref{lem:lattice-adjoints-full} proves
  \[
    T^{\flat}(x):=f_{\flat}\circ T_{\e} \circ f(x) \geq f_{\flat} \circ f(x) \geq x
  \]
  and hence that $T^{\flat}_{\e}$ is a translation too.
  Superlinearity needs to be checked.
  By definition
  \[
    T^{\flat}_{\e_2}\circ T^{\flat}_{\e_1} = f_{\flat}\circ T_{\e_2}\circ f \circ f_{\flat} \circ T_{\e_1} \circ f
  \]
  Now if $f$ respects joins, then \cref{prop:lattice-adjoint-joins} implies that $f \circ f_{\flat} \leq \id_{\pQ}$.
  Consequently
  \[
    T^{\flat}_{\e_2}\circ T^{\flat}_{\e_1} \leq f_{\flat} \circ T_{\e_2}\circ T_{\e_1}\circ f \leq f_{\flat} \circ T_{\e_2+\e_1} \circ f =:T^{\flat}_{\e_2+\e_1},
  \]
  which proves superlinearity of $T^{\flat}_{\bullet}$.
\end{proof}

\begin{rmk}[Upper Approximation Translation?]\label{rmk:upper-translation}
  In light of \cref{lem:lower-approximate-translation}, one can ask whether there is a dual story to be told.
  Suppose we let
  \[
    T^{\sharp}_{\e}:=f_{\sharp} \circ T_{\e} \circ f.
  \]
  If $T_{\e}$ is a translation then we get the awkward zig-zag of inequalities
  \[
    f_{\sharp}\circ T_{\e} \circ f(x) \geq f_{\sharp}\circ f(x) \leq x.
  \]
  In the event that $f$ is \emph{full} we get from \cref{lem:lattice-adjoints-full} the statement that $T^{\sharp}_{\e}$ is a translation on $\pL$:
  \[
    f \text{ is full } \Rightarrow T^{\sharp}_{\e}(x) \geq x
  \]
  However, superlinearity fails for $T^{\sharp}_{\bullet}$.
\end{rmk}

The lower approximation provides a simplified expression for shifting a module over $\pL$ when $\pQ$ is equipped with a superlinear family of translations.
The following corollary of \cref{prop:pushforward-join} and \cref{lem:lower-approximate-translation} allows us to dispense with many of the complications of \cref{defn:shift-relative-to-f} and work with an ``ordinary'' weak interleaving theory defined by $T^{\flat}_{\bullet}$.

\begin{cor}\label{prop:shift-inner}
  Fix $\pL$ a complete join semilattice, a poset $\pQ$ equipped with a superlinear family of translations $T_{\bullet}$ and let
  $f\colon\pL \rightarrow \pQ$ be a map of posets that respects joins, as defined in \cref{defn:respects-meets-joins}.
  The $\e$-shift of an $\pL$-module $M$ relative to $f$, as defined in \cref{defn:shift-relative-to-f}, can be re-expressed as pulling back along the lower approximation $T^{\flat}_{\e}$, defined in \cref{lem:lower-approximate-translation}.
  Said symbolically
  \[
    (M)^{\e}_{\pL} := f^*T^*_{\e}f_* M \cong T^{\flat *}_{\e}M.
  \]
  This isomorphism is natural in $M$, allowing us to replace the shift structure over $\pL$ with the shift structure defined by $T^{\flat}_{\bullet}$.
\end{cor}
\begin{proof}
  We apply \cref{prop:pushforward-join} to simplify the computation of the $\e$-shift of $M$ relative to $f$.
  \begin{eqnarray*}
    (M)_{\pL}^{\e} &:=& (f^*T^*_{\e}f_*M) \\
    &=& \Lan_f M \circ T_{\e} \circ f \\
    &\cong & M \circ f_{\flat} \circ T_{\e} \circ f \\
    &=:& M\circ T_{\e}^{\flat} \\
    &=:& T^{\flat *}_{\e}M
  \end{eqnarray*}
\end{proof}

\subsection{The Delta Approximation Condition}

We now isolate the condition needed to bound the distance between a $\pQ$-module and its upper and lower pixelizations.
These are some of the main results of this paper and generalize results of~\cite{Munch2016}.

\begin{defn}\label{defn:delta-approx}
  Suppose $\pQ$ is a poset equipped with a superlinear family of translations $T_{\bullet}$ and $f\colon\pP\to\pQ$ is a map of posets.
  We say that $f\colon\pP\to \pQ$ is a \define{$\delta$-approximation} if for every $q\in\pQ$ there exists a $p\in\pP$ so that $q\leq f(p) \leq T_{\delta}(q)$.
\end{defn}

The $\delta$-approximation condition allows us to bound the interleaving distance between a $\pQ$-module and its pixelizations.
Before proving this, we show how \cref{prop:pushforward-join} implies that both pixelizations can be thought of as pullbacks along the two possible Galois connections.

\begin{cor}\label{cor:upper-lower-pixelization}
  If $f\colon\pL\to\pQ$ respects joins, then for any $\pQ$-module $M$ the lower pixelization of $M$ can be re-expressed as
  \[
    f_*f^* M \cong f^*_{\flat}f^*M = (f\circ f_{\flat})^* M = M\circ f\circ f_{\flat}.
  \]
  Dually if $f\colon\pL\to\pQ$ respects meets, then for any $\pQ$-module $M$ the upper pixelization of $M$ can be re-expressed as
  \[
    f_{\dagger}f^* M \cong f^*_{\sharp}f^*M = (f\circ f_{\sharp})^* M = M\circ f\circ f_{\sharp}.
  \]
\end{cor}

We now state our main bounds between a $\pQ$-module and its pixelizations.

\begin{thm}\label{thm:delta-pixelization}
  Suppose $\pL$ is a complete lattice and $\pQ$ is a poset with a superlinear family of translations.
  If $f\colon\pL \to \pQ$ is a $\delta$-approximation that respects joins, then the interleaving distance between any $\pQ$-module $M$ and its lower pixelization is bounded above by $\delta$, i.e.
  \[
    \forall M\in \Fun(\pQ,\cat) \qquad d_{\pQ}(M,f_*f^*M) = d_{\pQ}(M,(f\circ f_{\flat})^*M) \leq \delta.
  \]
  Dually, if $f\colon\pL\to \pQ$ is a $\delta$-approximation that respects meets, then
  \[
    \forall M\in \Fun(\pQ,\cat) \qquad d_{\pQ}(M,f_{\dagger}f^*M) = d_{\pQ}(M,(f\circ f_{\sharp})^*M) \leq \delta.
  \]
\end{thm}
\begin{proof}
  We know from \cref{prop:lattice-adjoint-joins} that if $f$ respects joins then $f\circ f_{\flat} \leq \id_{\pQ}$.
  This implies there are natural morphisms
  \[
    (f\circ f_{\flat})^*M = M\circ f\circ f_{\flat} \to M \to (M)^{\delta}_{\pQ}.
  \]
  The above composition participates in a $\delta$-interleaving as we now show.
  If we consider $(M\circ f\circ f_{\flat})^{\delta}_{\pQ}$ this unravels to $M\circ f\circ f_{\flat}\circ T_{\delta}$.
  Now we note that the $\delta$-approximation condition implies that
  \[
    \forall q\in \pQ \qquad \exists x\in\pL \qquad \text{s.t} \qquad q\leq f(x) \leq T_{\delta}(q).
  \]
  Such an $x$ above is necessarily less than or equal to $f_{\flat}(T_{\delta}(q))$.
  Since $f$ is a poset map that preserves joins we have that
  \[
    \forall q\in \pQ \qquad q\leq f(f_{\flat}(T_{\delta}(q))) \leq T_{\delta}(q).
  \]
  This induces a sequence of natural morphisms
  \[
    M \to (f^*_{\flat}f^*M)^{\delta}_{\pQ} \to (M)^{\delta}_{\pQ}.
  \]
  It is left to the reader to check that these participate in an interleaving, thereby proving that
  \[
    d_{\pQ}(M,f_*f^*M) = d_{\pQ}(M,f^*_{\flat}f^*M)\leq \delta.
  \]

  For the upper pixelization, the proof is quite similar, but we go through the necessary invocations.
  We know from \cref{prop:lattice-adjoint-joins} that if $f$ respects meets then $f\circ f_{\sharp} \geq \id_{\pQ}$.
  This implies that there is a natural morphism from $M$ to its pixelization
  \[
  M \to M \circ f\circ f_{\sharp} = f^*_{\sharp}f^* M \to (f^*_{\sharp}f^*M)^{\delta}_{\pQ}.
  \]
  To construct the other morphism that participates in a $\delta$-interleaving we note that the $\delta$-approximation condition implies
  \[
    \id_{\pQ} \leq f\circ f_{\sharp} \leq T_{\delta}.
  \]
  Post-composing with $M$ provides the second necessary morphism from the upper pixelization to the $\delta$-shift of $M$.
  \[
    (f\circ f_{\sharp})^*M = M\circ f\circ f_{\sharp} \to M\circ T_{\delta}=:(M)^{\delta}_{\pQ}
  \]
  It is left to the reader to check that this defines an interleaving and hence
  \[
    d_{\pQ}(M,f_{\dagger}f^*M) = d_{\pQ}(M,f^*_{\sharp}f^*M) \leq \delta.
  \]
\end{proof}

\cref{thm:delta-pixelization} proves that we can infer the true interleaving distance between two $\pQ$-modules using either pixelization.

\begin{cor}\label{cor:pixelization-distortion}
  Suppose $\pL$ is a complete lattice and $\pQ$ is a poset with a superlinear family of translations.
  If $f\colon\pL \to \pQ$ is a $\delta$-approximation that respects joins, then the interleaving distance between two $\pQ$-modules $M$ and $N$ can be inferred using their lower pixelizations.
  Specifically
  \[
    |d_{\pQ}(M,N) - d_{\pQ}(f_*f^*M,f_*f^*N)|\leq 2\delta
  \]
  Dually, if $f\colon\pL \to \pQ$ is a $\delta$-approximation that respects meets, then the interleaving distance between two $\pQ$-modules $M$ and $N$ can be inferred using their upper pixelizations.
  \[
    |d_{\pQ}(M,N) - d_{\pQ}(f_{\dagger}f^*M,f_{\dagger}f^*N)|\leq 2\delta
  \]
\end{cor}
\begin{proof}
  Both statements follow from the triangle inequality and \cref{thm:delta-pixelization}.
\end{proof}

The lower pixelization has the advantage that it plays well with the relative interleaving distance on $\pL$.
This allows us to state how much distortion the pullback functor has.

\begin{thm}\label{thm:delta-distortion}
  If $\pL$ is a complete lattice, $\pQ$ is equipped with a superlinear family of translations $T_{\bullet}$ and $f\colon\pL \to \pQ$ is a $\delta$-approximation that respects joins, then the distortion of the pullback functor
  \[
    f^*\colon \Fun(\pQ,\cat) \to \Fun(\pL,\cat) \qquad M \squigrightarrow f^*M
  \]
  is at most $2\delta$.
  In other words for any pair of $\pQ$-modules $M$ and $N$ the difference in interleaving distance is bounded by
  \[
    |d_{\pQ}(M,N)-d_{\pL}(f^*M,f^*N)|\leq 2\delta.
  \]
  Note here that the interleaving distance over $\pL$ is the weak relative interleaving distance of \cref{defn:relative-weak-interleaving}, which by virtue of \cref{prop:shift-inner} can be identified with a weak interleaving distance over $\pL$ that is defined using the lower approximation of $T_{\bullet}$, i.e. $T^{\flat}_{\bullet}$.
\end{thm}

\section{Applications to Mapper and Cosheaves}
\label{sec:mapper}

Much of our efforts have been inspired by the fundamental convergence result of~\cite{Munch2016}.
This result shows how one can use the algorithm of \texttt{Mapper}~\cite{Singh2007} to reliably approximate the Reeb cosheaf associated to a map $g\colon \topY \to \topX$ where $\topX$ is a metric space equipped with a cover $\calU$.
The way that \texttt{Mapper} works is that it defines a cellular cosheaf over the nerve of $\calU$ that associates to each simplex $\sigma\in N(\calU)$ the set of path components $\pi_0(g^{-1}(U_{\sigma}))$.
The Reeb cosheaf, by contrast, associates to every open set $U\in\Open(\topX)$ the set of path components $\pi_0(g^{-1}(U))$ of the pre-image.
The content of~\cite{Munch2016} is that when the cover $\calU$ is by open sets with diameter at most $\delta$, then one can convert the cellular cosheaf over the nerve into a cosheaf on $\topX$ and this cosheaf has interleaving distance at most $\delta$ with the Reeb cosheaf.
In this sense,~\cite{Munch2016} proves the correctness and stability of the~\texttt{Mapper} algorithm.

We begin by showing how to complete a cover to a lattice so that the results of \cref{sec:lattice-approximation} can be used to approximate the Reeb cosheaf, or any cosheaf, for that matter.
We apply the relative interleaving theory developed in \cref{sec:relative-interleavings} to a general approximating cover and prove our main interleaving inference result for cosheaves in \cref{thm:cosheaf-interleaving-approximation}.
This result has a similar flavor to the one found in~\cite{Munch2016}, although the lower-pixelization of the Reeb cosheaf differs from the construction used there.
Additionally our result focuses on how one can define interleavings intrinsicly over the meet completion of a cover.
This general discussion is finally specialized to the study of a cover that is derived from a cellular structure on $\R^n$ where we compute an explicit weighting on the face relation poset that gives rise to the lower approximation translation of \cref{lem:lower-approximate-translation}.
In this setting one can see how the computation of interleavings can be inferred using finitely many computations.

\subsection{Covers, Lattices and Cosheaves}

Any topological space $\topX$ naturally has a complete lattice associated to it, namely the poset of open sets $\Open(\topX)$.
The join of a collection of open sets is obviously the union of those open sets.
The join of the empty collection is the empty set.
By contrast, the meet of a collection of empty sets is the interior of the intersection.
The meet of the empty collection is the set $\topX$.

Suppose now that $\cU$ is a cover of $\topX$.
In order to apply the theory of \cref{sec:lattice-approximation} we need to associate to our cover $\calU$ three algebraic devices: a complete meet semilattice, a complete join semilattice and a complete lattice.

\begin{defn}[Semilattices Associated to a Cover]\label{defn:cover-semilattices}
  Suppose $\calU=\{U_k\}_{k\in\Lambda}$ is a cover of a topological space $\topX$.
  We define the \define{meet completion} of $\cU$, written $\cM(\cU)$, to have elements given by
  \[
    \bigwedge_{k\in \sigma} U_k = \text{int}\left(\bigcap_{k\in \sigma} U_k\right) =: \text{int}\left(U_{\sigma}\right)
  \]
  Note that the meet completion automatically includes as a full and faithful subposet of $\Open(\topX)$.

  The \define{join completion} of a cover $\cU$, written $\cJ(\cU)$, has elements given by
  \[
    \bigvee_{k\in \sigma} U_k=\bigcup_{k\in \sigma} U_k.
  \]
  for any $\sigma\subseteq \Lambda$.

  Finally, we define the \define{lattice completion} of $\calU$, written $\pL(\cU)$ to be the join completion of $\cM(\cU)$.
  This includes all unions of intersections of elements of $\cU$, i.e. every element of $\pL(\cU)$ is of the form
  \[
    \bigvee_{\sigma \in N(\cU)} U_{\sigma}.
  \]
  The lattice completion $\pL_{\cU}$ clearly includes as a full and faithful subposet of $\Open(\topX)$.
\end{defn}

We note that the meet completion, the join completion and lattice completion of a cover all participate in the following diagram of full and faithful inclusions of posets.
\[
  \begin{tikzcd}
    & \cJ(\cU) \ar[dr,"m"'] \ar[drr,"u"] & & \\
    \cU \ar[ur] \ar[dr] & & \pL(\cU) \ar[r,"f"] & \Open(\topX) \\
    & \cM(\cU) \ar[ur,"j"] \ar[urr,"i"'] & &
  \end{tikzcd}
\]
We will investigate each of the named arrows above in turn, but will start with the maps $f_{\flat}$ and $f_{\sharp}$ associated to $f$ as defined in \cref{prop:lattice-adjoint-joins} in this setting.

\begin{prop}\label{prop:open-flat-sharp}
  If $f\colon\pL(\cU) \to \Open(X)$ is the lattice completion of a cover $\cU$, then
  \[
    f_{\flat}(V) = \sup\{ \cup_{\sigma} U_{\sigma} \in \pL_{\cU} \mid \cup_{\sigma} U_{\sigma} \subseteq V\}.
  \]
  Notice that for typically small open sets $V$, $f_{\flat}(V)=\varnothing$.
  Since $\Open(X)$ is a complete lattice, we note that $f_{\flat}$ preserves meets.

  By contrast, $f_{\sharp}$ is non-trivial for every non-empty open set $V\in\Open(X)$.
  \[
    f_{\sharp}(V) = \inf\{\cup_{\sigma} U_{\sigma} \in \pL(\cU) \mid V \subseteq \cup_{\sigma} U_{\sigma}\}
  \]
  Finally, we note that $f_{\sharp}$ preserves joins, i.e. unions are sent to unions.
\end{prop}
\begin{proof}
  Most of the above statements are just from the definitions of \cref{prop:lattice-adjoint-joins}.
  The fact that $f_{\flat}$ commutes with meets and that $f_{\sharp}$ commutes with joins are consequences of the Adjoint Functor Theorem.
  Since $f$ preserves joins (colimits), it is a left adjoint and hence its right adjoint $f_{\flat}$ preserves meets (limits).
  Additionally since $f$ preserves meets (limits), it is also a right adjoint and hence its left adjoint $f_{\sharp}$ preserves joins (colimits). See Proposition 1.104 of~\cite{fong2019invitation} for a reference.
\end{proof}

Before we consider the other poset maps indicated above, we proceed to directly apply our relative interleaving theory and the approximation results obtained above.

\subsection{Interleaving Isometry and Approximation Results}

We are now in a position to apply all of the theory developed in \cref{sec:relative-interleavings} and \cref{sec:approximation}, but we first provide a natural example of a superlinear family of translations on $\Open(\topX)$.

\begin{defn}[Metric Translation]\label{defn:metric-translation}
  Suppose $\topX$ is a metric space with metric $d$.
  The \define{metric translation} $T\colon\Open(\topX)\times [0,\infty) \to \Open(\topX)$ is defined by
  \[
    T_{\e}(U)=U^{\e} \qquad \text{where} \qquad U \squigrightarrow U^{\e}=\cup_{x\in U} B(x,\e).
  \]
  Here $B(x,\e)$ is the open ball of radius $\e$ about $x$.
\end{defn}

First we notice that if $\Open(\topX)$ is equipped with a superlinear family of translations $T_{\bullet}$, then by \cref{prop:shift-inner} the relative weak interleaving distance $d_f$ defined on $\pL(\calU)$
can be identified with an ordinary weak interleaving distance defined using the lower approximation to $T$, which we called $T^{\flat}_{\bullet}$ i.e.~$d_{\pL(\cU)}$.
In other words
\[
  d_{f}(-,-)=d_{\pL(\cU)}^{T^{\flat}}(-,-).
\]
We can use the superlinear family of lower approximations $T^{\flat}_{\bullet}$ to then define a relative weak interleaving distance on $\cM(\calU)$, written $d_{j\colon\cM(\calU) \to\widebar{\calU}}$ or $d_{\cM(\calU)}$ for short.
Our main corollary of \cref{thm:extend-restrict-interleavings} can then be phrased as follows.

\begin{cor}\label{cor:cover-completion-isometry}
  Suppose $\Open(\topX)$ is equipped with a superlinear family of translations $T_{\bullet}$.
  If $\calU$ is a cover of $\topX$ and $\pL(\cU)$ is the lattice completion of $\calU$, then let $T^{\flat}_{\bullet}$ be the lower approximation of $T_{\bullet}$ by $f\colon\pL(\cU)\hookrightarrow \Open(\topX)$.
  Denoting the inclusion of the meet completion of $\cU$ into  by $i\colon\cM(\calU) \hookrightarrow \Open(\topX)$ and $j\colon\cM(\calU) \hookrightarrow \pL(\cU)$, we have that if $M$ and $N$ are $\Open(\topX)$-modules, then
  \[
    d_{\cM}(i^*M,i^*N) = d_{\pL}(j_*i^*M,j_*i^*N).
  \]
  In other words, we can define an interleaving theory over the meet completion $\cM(\cU)$ and an interleaving theory over the lattice $\pL(\cU)$ and the pushforward functor along $j$ defines an isometry between these two categories.
\end{cor}

We would like to say that the interleaving distances above can be used to infer something about the interleaving distance of $M$ and $N$ over $\Open(\topX)$.
However this requires the cosheaf axiom~\cref{defn:cosheaf} as an additional assumption.

\begin{prop}\label{prop:cosheaf-cover-identity}
  If $M$ is a cosheaf on a topological space $\topX$, then using the notation of \cref{cor:cover-completion-isometry} we have
  \[
    f^*M\cong j_*i^*M
  \]
\end{prop}
\begin{proof}
  Suppose $U$ is the union of some cover elements $\{U_k\}_{k\in\sigma}$ or their intersections.
  Let $\mathcal{C}_U$ be those cover elements along with their intersections.
  By the cosheaf axiom we know that
  \[
    \varinjlim_{i\in \mathcal{C}_U} M(U_i) \cong M(\bigvee \mathcal{C}_U) = M(U).
  \]
\end{proof}

When $\pL(\cU)$ is a $\delta$-approximation of $\Open(\topX)$ then we get the desired inference theorem for the interleaving distance between cosheaves.

\begin{thm}\label{thm:cosheaf-interleaving-approximation}
  If $\pL(\cU)$ is the lattice completion of a cover of a topological space $\topX$ and $T_{\bullet}$ is a superlinear family of translations on $\Open(\topX)$ where
  \[
    \forall V\in\Open(\topX) \quad \exists U\in \pL(\cU) \qquad \text{with} \qquad V \subseteq U \subseteq T_{\delta}(V)
  \]
  for a fixed $\delta>0$, then for any pair of cosheaves $M,N\in\CoShv(X)$ we have that
  \[
    |d_{\cM}(i^*M,i^*N) - d(M,N)|=|d_{\pL}(f^*M,f^*N) - d(M,N)| \leq 2\delta.
  \]
\end{thm}

\subsection{Intrinsic Interleavings for Grid Opens}
\label{ssec:grid-opens}

In this section we consider a complete lattice $\pL$ that serves as a $\delta$-approximation of $\Open(\R^n)$ where the relative interleaving distances over $\pL$ can be computed explicitly.
This is done by considering a cell structure on $\R^n$, whose cover by open stars gives rise to a cover whose meet completion is isomorphic to the face relation poset of this cell structure.
Our approximating lattice will be unions of open stars of cells, which we call ``grid opens.''
Our superlinear family of metric translations pulls back to a translation that comes from a weighting on the face relation poset of our cell structure.

We begin by considering an explicit cell structure on $\R^n$.
Let
\[
  \Lambda:= \{\delta (k_1,\ldots,k_n)\in\R^n \mid k_i\in \mathbb{Z}\}
\]
be the geometric lattice\footnote{This lattice is not to be confused with a poset that is equipped with a meet and join operation.} generated by a $\delta$-scaling of the standard basis $\{e_i\}$ for $\R^n$.
Associated to this is regular cell complex where each cell $|\sigma|$ is homeomorphic to an open cube $(0,1)^{d}$ of some dimension $d\in\{0,\ldots,n\}$; note that we assume that $(0,1)^0$ is a one point space.
We will relate this cell structure to the meet completion of a particular cover of $\R^n$.

\begin{defn}
  Let $x\in \Lambda$ be a point in the above geometric lattice.
  Let
  \[
    U_{x}:=\{(y_1,\ldots,y_n) \mid \max_{i} |x_i-y_i|<\delta\}
  \]
  be the $\delta$-ball around $x$ measured in the sup norm $||\cdot||_{\infty}$.
  Let $\cU_{\Lambda}$ be the collection of such $\delta$-balls adapted to the lattice $\Lambda$.
\end{defn}

Following \cref{defn:cover-semilattices}, we now consider the meet completion of $\cU_{\Lambda}$, written $\cM(\cU_{\Lambda})$.
Note that the elements of $\cM$ are given by open sets of the form
\[
  U_{\sigma} := \cap_{x\in \sigma} U_x
\]
for $\sigma\subseteq \Lambda$.
We now relate the meet completion to the obvious cell structure on $\R^n$ induced by $\Lambda$.

\begin{prop}
  The geometric lattice $\Lambda$ defined above induces an obvious cubical cell structure on $\R^n$.
  Let $\Cell(\Lambda)$ be the \define{face relation poset} of this cell structure, i.e. we say that two cells $\sigma\leq \tau$ if there is a reverse containment of their closures: $\bar{|\sigma|} \supseteq |\tau|$.
  With this choice of partial order the map
  \[
    \text{star}\colon \Cell(\Lambda) \to \cM(\cU_{\Lambda}) \qquad \sigma \squigrightarrow \text{star}(|\sigma|)
  \]
  is an isomorphism onto its image.
\end{prop}
\begin{proof}
  If we consider a cell $|\sigma|$ then then let $v(\sigma)$ denote the vertices in its closure.
  It's easy to see that that the open star of the cell $|\sigma|$ is equivalently viewed as
  \[
    U_{v(\sigma)}=\bigcap_{x\in v(\sigma)} U_x.
  \]
  The only reason we say an isomorphism onto its image is that $\cM(\cU_{\Lambda})$ includes the empty set, whereas $\Cell(\Lambda)$ has no element corresponding to that.
\end{proof}

We now give a geometric characterization of the lattice completion of $\cU_{\Lambda}$ in terms of the face relation poset $\Cell(\Lambda)$.
Note that since the only element not witnessed by $\cM(\cU_{\Lambda})$ is the empty set, the join completion of the image of $\Cell(\Lambda)$ under the open star map will equal the join completion of $\cM(\cU)$.

\begin{defn}\label{def:GridO}
  The poset of \define{grid opens}, denoted $\GridO^n_\delta$, or simply $\GridO$ when the parameters are clear, is any open set of the form
  \[
    U = \bigcup_{\sigma\in\Cell(\Lambda)} \text{star}(|\sigma|).
  \]
  We also declare that the empty set is a grid open.
  Let
  \[
    f\colon \GridO^n_{\delta} \to \Open(\R^n)
  \]
  denote the obvious inclusion. For a visualization of some grid opens, we refer to \cref{fig:CellComplexExamples}.
\end{defn}

In order to proceed with our approximation theory, we now check that $f$ is a $\delta$-approximation.

\begin{lem}
  \label{lem:GridO_DeltaApprox}
  Let $T_{\bullet}$ denote the metric translation of $\Open(\R^n)$ with respect to he $\ell_{\infty}$ norm.
  The map $f\colon \GridO_\delta^n \to \Open(\R^n)$ is a $\delta$-approximation.
\end{lem}

\begin{proof}
Given $V \in \Open(\R^n)$, let
\[
  U=\bigcup_{|\sigma|\cap V \neq \varnothing} \text{star}(|\sigma|)
\]
be the union of the open stars of those cells with non-trivial intersection with $V$.
Since the interiors of each cell partition $\R^n$, we obviously have that $V\subseteq U$.
Moreover any point $y \in U$, we know $y \in |\sigma|$ for some $\sigma$ with $|\sigma| \cap U \neq \emptyset$.
Say $x \in |\sigma| \cap U$ and note that $\|y-x\|_\infty \leq \delta$.
Thus $\|x-V\|_\infty \leq \delta$, so $U \subseteq V^\delta=T_{\delta}(V)$.
\end{proof}

\begin{figure}
  \includegraphics[width = .8\textwidth]{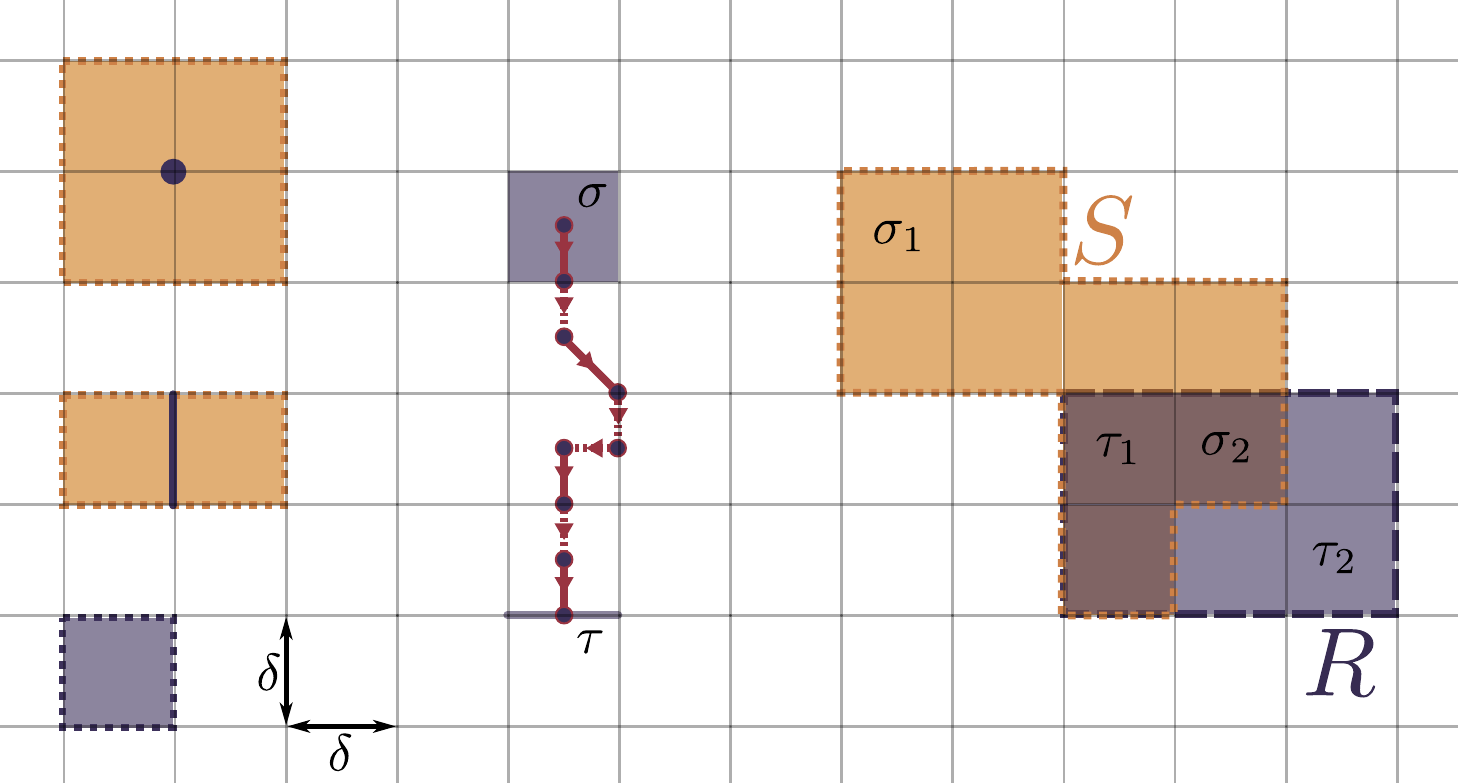}
  \caption{
  \textit{Left column:} examples of principal downsets in $\GridO_\delta^2 = \Down(\Cell_\delta^2)$ for a 0-, 1-, and 2-cell.
  \textit{Middle:} a path from the top 2-cell $\sigma$ to the bottom 1-cell $\tau$ where solid red edges constitute a pair with $w(\gamma_i,\gamma_{i+1}) = \delta$, and dashed where it is 0.
  In fact, this path achieves the weighting $w(\sigma,\tau) = 4\delta$ even though it does not have minimum length.
  \textit{Right:} Given $S,R \in \GridO$, $w_\infty (S,R) = w(\sigma_1, \tau_1) = 2\delta$, and $w_\infty(R,S) = w( \tau_2, \sigma_2) = \delta$.
  }
  \label{fig:CellComplexExamples}
\end{figure}

Applying \cref{thm:cosheaf-interleaving-approximation} to $\GridO$ shows that we can approximate the interleaving distance between two cosheaves over $\R^n$ using only open stars of cells defined by $\Lambda$.

\begin{cor}\label{cor:delta-grid-approximation}
  Let $\Open(\R^n)$ be equipped with the $\ell_{\infty}$-metric translation $T_{\bullet}$.
  Let $\Cell(\Lambda)$ be the cell structure induced by the $\delta$-net of points described above and let $\GridO$ denote the set of grid opens in $\R^n$.
  For any pair of cosheaves $M,N\in\CoShv(\R^n)$ we have that
  \[
    |d_{\Cell(\Lambda)}(i^*M,i^*N) - d(M,N)|=|d_{\GridO}(f^*M,f^*N) - d(M,N)| \leq 2\delta.
  \]
\end{cor}

To better appreciate this result, we consider a more combinatorial description of the lower approximation translation $T^{\flat}$ of grid opens induced from $\Open(\R^n)$.
In the setting of $n=1$ we can use the upper half plane to visualize how connected open sets are translated by $T$ and then lower approximated by $T^{\flat}$.

For $n\geq 2$ visualizations are not so accessible if one tries to view all grid opens at once.
Instead we can use the weighted poset framework of \cref{defn:weighted-poset}.

\begin{defn}\label{defn:weighted-cell-structure}
Let $w\colon\Cell_\delta^n \times \Cell_\delta^n \to \R$ denote a weighting that only takes integer multiples of $\delta$.
We define this weight as follows:
\begin{itemize}
	\item The Lawvere distance from a cell to any of its cofaces is 0, i.e. if $\tau\leq \sigma$ then we set $w(\sigma,\tau) = 0$.
	\item The Lawvere distance from a celll to any of its faces is $\delta$, i.e. if $\sigma \leq \tau$, set $w(\sigma,\tau) = \delta$.
	\item If $\sigma$ and $\tau$ are incomparable, we set $w(\sigma,\tau)$ to be the weight of the shortest length path in $\Cell_\delta^n$. Here a ``path'' $\gamma$ from $\sigma$ to $\tau$ in $\Cell_\delta^n$ is a zig-zag of comparable elements $\sigma = \gamma_0, \gamma_1,\gamma_2,\cdots,\gamma_k=\tau$. In this setting
	\[
		w(\sigma,\tau) = \inf_{\gamma\colon \sigma \squigrightarrow \tau} \sum w(\gamma_i,\gamma_{i+1}).
	\]
\end{itemize}
\end{defn}

It is immediate to check that $w$ constitutes a weighting on the poset $\Cell_\delta^n$.
Recall that the directed ball associated to this weighting is given by
\[
  \overrightarrow{B}_w(\sigma;\epsilon) = \{\rho \in \Cell_\delta^n \mid w(\sigma, \rho) \leq \e \}
\]
In particular
\[
  \overrightarrow{B}_w(\sigma;0) = \text{star}(|\sigma|).
\]
To simplify this notation and make sure that we always associate a grid open to a cell $\sigma$ and a non-negative real number $\e\geq 0$, we set
\[
  \text{star}(|\sigma|;\e):=\bigcup \text{star}(|\rho|) \qquad \text{for} \qquad \rho\in \overrightarrow{B}_w(\sigma;\epsilon)
\]
The following lemma is straightforward, albeit tedious to check.

\begin{lem}\label{lem:grid-lower-translation}
The lower approximation translation $T^{\flat}_{\bullet}$ of the $\ell_{\infty}$-metric translation $T$ can be internally characterized by the weighting $w$ on $\Cell^n_{\delta}$ from \cref{defn:weighted-cell-structure} by the formula
\[
  T^{\flat}\colon\GridO\times [0,\infty) \to \GridO \qquad U=\bigcup \text{star}(|\sigma|) \squigrightarrow U^{\e}= \bigcup \text{star}(|\sigma|;\e).
\]
Notice that $T^{\flat}$ only changes at integer multiples of $\delta$, i.e.~$T^{\flat}_\e = T^{\flat}_{\lfloor{\e/\delta}\rfloor\delta}$.
\end{lem}

\begin{rmk}\label{rmk:the-utility-of-cosheaves}
The upshot of \cref{cor:delta-grid-approximation} and \cref{lem:grid-lower-translation} is that if $M$ and $N$ are cosheaves over $\R^n$ that are finitely interleaved, then one can infer this interleaving distance using only finitely many computations as well as by only knowing the values of $M$ and $N$ on open stars of cells in $\Lambda$.
This gives strong locality and finiteness results for interleavings of cosheaves defined on $\R^n$.
Of course \cref{thm:delta-distortion} and \cref{lem:grid-lower-translation} also demonstrate that if two arbitrary $\Open(\topX)$-modules $M$ and $N$ are finitely interleaved then we can discover this approximate interleaving distance using only finitely many computations, but with the caution that one must have access to the values of $M$ and $N$ on arbitrary grid opens.
This may be an unrealistic assumption and displays the utility of working with topological summaries that are cosheaves.
Additionally, existing work shows that in many cases computation of the interleaving distance is NP-hard, see~\cite{bjerkevik_et_al:LIPIcs:2018:8726} and the references therein.
\end{rmk}

\section{Conclusion}

In this paper we considered how to define interleavings of $\pP$-modules using an interleaving theory over $\pQ$ and a map of posets $f:\pP \to \pQ$.
Our motivation for taking up this question is that many of the interleaving theories discussed in~\cite{Bubenik2014a,stefanou2018,de2018theory} use the notion of a superlinear family of translations and many natural posets, such as those of zig-zag type, do not have any non-trivial translations.
A correction to this deficiency is hinted at in~\cite{botnan2018algebraic}, where one embeds a zig-zag poset $\pZ$ into $\R^2$ and uses translations defined over $\R^2$, which is extended in this paper.
We observed here that a more intrinsic rescue from this problem is given by considering a superlinear family of translations on down sets in a poset $\pP$.

In order to develop a fully general, relative interleaving theory, we find that if one pushes forward a module on $\pP$ to define a $\pQ$-module and then uses translation operations over $\pQ$, then upon restriction one has defined a notion of a ``shift relative to $f$.''
We observed in \cref{lem:full-restricts-interleavings} that using the notion of weak (or pentagonal) interleavings allows us to always restrict an interleaving over $\pQ$ to get an interleaving, relative to $f$, over $\pP$.
This was crucial to our main theorem that allowed us to ignore any differences between interleaving over $\pQ$ or ``intrinsic'' interleavings over $\pP$.
This proof seemed to rely on the extra arrows that are used in the definition of a weak interleaving.
This leads us to our first question:
\begin{center}
  {\bf Open Question 1:}
\end{center}
\begin{quote}
  Does \cref{lem:full-restricts-interleavings} hold when we work with strict interleavings over $\pP$ and $\pQ$, rather than weak interleavings? Such an analysis would reveal further theoretical properties that weak interleavings of~\cite{de2018theory} enjoy over strict interleavings.
\end{quote}

We then proceeded to use this relative interleaving theory to prove that with an extra $\delta$-approximation condition, one can use the interleaving distance over $\pP$ to infer the interleaving distance over $\pQ$.
The motivation for this was two-fold:
\begin{enumerate}
  \item We want to be able to ``pixelate'' a module defined over a continuous poset $\pQ$ using a discrete poset $\pP$ and guarantee that this pixelation is not far off from the original module.
  \item We'd like to be able to perform interleaving inference.
\end{enumerate}
Although, we had two notions of pixelization---upper and lower---we mostly used the lower one because it allowed us to define a super-linear family of translations of $\pP$-modules.
In fact, \cref{rmk:upper-translation} showed that the upper approximation translation does not have this property.
This leads us to our second question:
\begin{center}
  {\bf Open Question 2:}
\end{center}

\begin{quote}
  Instead of using the left Kan extension along a map of posets $f\colon\pP\to\pQ$ and a superlinear family of translations over $\pQ$, what would happen if we chose to develop a relative interleaving theory that interacted well with the right Kan extension $f_{\dagger}$?
  Some preliminary work suggests that the \emph{upper} pixelization functor might have better algebraic properties than the lower pixelization functor, especially when studying Reeb cosheaves.
  The deficiency with using the upper pixelization is that currently no existing interleaving theory directly accounts for the reversed arrows one would need to introduce to define interleavings properly in this setting.
  Presumably, one could dualize the theory of~\cite{stefanou2018} to define a distance using an \emph{op}-lax monoidal functor from $[0,\infty)$ to $\Trans_{\pQ}$, but then one would need to check that things like the triangle inequality hold.
\end{quote}

Finally, we conclude that our study suggests that further study of sheaves and cosheaves over a metric space is necessary, especially if one wants to extend the theories of sampling and inference outlined in~\cite{robinson2015sheaf} and here.
Although plenty of work is already under way, we believe the study of algebraic structures at varying scales is a fruitful area of research with interesting theoretical and practical components.

\section*{Acknowledgements}
MBB has been partially supported by the DFG Collaborative Research Center SFB/TR 109 ``Discretization in Geometry and Dynamics''.
The work of JC was supported in part by NSF Grant No.~CCF-1850052.
JC would also like to thank Hans Riess for answering questions about lattice theory, which greatly improved the last two sections of this paper.
The work of EM was supported in part by NSF Grant Nos.~NSF CCF-1907591, DMS-1800446 and CMMI-1800466.

\appendix

\section{Review of Kan Extensions, Cofinality and Cosheaves}

In this section we provide a more detailed recollection of Kan extensions and cofinality.
We also provide a recollection of the statement that the left Kan extension of a module along the full and faithful inclusion into its lattice of down sets produces a cosheaf, which is the content of \cref{thm:kan-cosheaves} and whose proof is detailed in~\cite{curry2019kan}.

\subsection{Comma Categories}

\begin{defn}\label{defn:comma-cat}
Suppose $E\colon\aat\to\bat$ is a functor and let $b$ be an object of $\bat$.
The \define{comma category under $b$}, written $(E\downarrow b)$, is defined as follows:
\begin{itemize}
	\item The objects of $(E\downarrow b)$ are morphisms in $\bat$ of the form $\alpha\colon E(a) \to b$  where $a$ is any object of $\aat$.
	\item A morphism of $(E\downarrow b)$ between two objects $\alpha\colon E(a) \to b$ and $\alpha'\colon E(a') \to b$ is a morphism $\gamma\colon a \to a'$ in $\aat$ making the following diagram commute:
	\[
		\xymatrix{E(a) \ar[rr]^{E(\gamma)} \ar[rd]_{\alpha} & & E(a') \ar[ld]^{\alpha'} \\ & b &}
	\]
\end{itemize}
There is also a \define{comma category over $b$}, written $(b \downarrow E)$, that is defined completely dually: objects are morphisms in $\bat$ of the form $\alpha\colon b \to E(a)$ for some $a$ in $\aat$, morphisms are morphisms in $\aat$ making the dual triangle commute:
    \[
    	\xymatrix{ & b \ar[dl]_{\alpha} \ar[dr]^{\alpha'} &  \\ E(a) \ar[rr]_{E(\gamma)} & & E(a') }
    \]
\end{defn}

For an example of a comma category, we consider the special case of maps between posets.

\begin{ex}
Recall that a map of posets $f\colon\pP\to\pQ$ is equivalently a functor.
Substituting $f$ for $E$ in the above definition leads to the following interpretations:
The comma category $(f\downarrow q)$ is simply the sub-poset of $\pP$ consisting of those $p$ such that $f(p)\leq q$, which one might call the ``sublevel set of $f$ at $q$.''
The comma category $(d\downarrow f)$ is thus the superlevel set of $f$ at $q$.
\end{ex}

\subsection{Kan Extensions}

The comma category $(E\downarrow b)$ associated to a functor $E\colon\aat\to\bat$ and an object $b$ in $\bat$, has a natural projection functor $\pi^b \colon (E\downarrow b) \to \aat$ that sends an object $\alpha\colon E(a) \to b$ to the object $a$ in $\aat$, a morphism $\gamma\colon a\to a'$ goes to the same morphism in $\aat$.
This observation, and this particular choice of comma category, allows us to define the \emph{left Kan extension} of a functor $F\colon \aat \to \cat$ \emph{along} the functor $E\colon \aat \to \bat$.

\begin{defn}[Pointwise Kan Extensions, cf.~\cite{riehl2017category} Thm. 6.2.1]\label{defn:pointwise-left-Kan}

The \define{left Kan extension} of $F\colon \aat \to \cat$ along $E\colon \aat \to \bat$ is a functor $\Lan_E F\colon \bat \to \cat$ that assigns to an object $b$ of $\bat$ the following colimit
\[
	\Lan_E F (b) = \varinjlim \left( (E \downarrow b) \xrightarrow[]{\pi} \aat \xrightarrow[]{F} \cat \right) = \varinjlim_{E(a) \rightarrow b} F(a)
\]
Morphisms are sent to corresponding universal maps between colimits.
\end{defn}

\begin{ex}
Suppose $j\colon \pP \hookrightarrow \pQ$ is an inclusion of posets and suppose $M\colon\pP \to \cat$ is a functor.
The left Kan extension of $M$ along $j$ assigns to an element $q\in Q$ the colimit of $M$ over the sublevel set of $j$ at $q$.
Note that this uses the fact that there there is at most one morphism of the form $j(p) \leq q$.
Moreover, if the inclusion is full, i.e.~if $p\leq_{\pP} p'$ if and only if $j(p)\leq_{\pQ} j(p')$, then the colimit can be viewed as occuring over all $p\in \pP$ such that $j(p)\leq q$.
\end{ex}

The following example is of utmost importance.

\begin{ex}
Let $\iota\colon \pP \hookrightarrow \Down(\pP)$ denote the map of posets that sends $p\in \pP$ to the principal downset $D_p$.
Let $S\in \Down(\pP)$ be an arbitrary downset.
The reader is asked to convince themselves that the comma category $(\iota \downarrow S)$ is given by the full subcategory of $\pP$ whose objects are those $p\in S$, which we write as $\pP_S$.
Consequently, if we wish to consider the left Kan extension of a $\pP$-module
$M\colon\pP \to \cat$,
then we have that
\[
	\iota_*M(S) := \Lan_{\iota} M(S) = \varinjlim \left( \pP_S \hookrightarrow \pP \to \cat \right) = \varinjlim_{p\in S} M(p).
\]
\end{ex}

The following lemma can be deduced from Proposition 6.1.5 of~\cite{riehl2017category}, but we include it here for the reader's convenience.

\begin{lem}\label{app:lem:push-pull-unit-counit}
  Let $f\colon\pP \to \pQ$ be a map of posets and let $f^*\colon\Fun(\pQ,\cat) \to \Fun(\pP,\cat)$ be the pullback and $f_*\colon\Fun(\pP,\cat) \to \Fun(\pQ,\cat)$ be the pushforward functors defined in \cref{defn:pullback} and \cref{defn:pushforward}.
  There are natural transformations
  \[
    \upsilon \colon \id_{\cat^{\pP}} \Rightarrow f^*f_* \qquad \text{and} \qquad \chi \colon f_*f^* \Rightarrow \id_{\cat^{\pQ}}
  \]
  called the \define{unit} and the \define{co-unit} of the adjunction that participate in the observation that $f_*$ is \define{left adjoint} to $f^*$.
\end{lem}
\begin{proof}
  The above statements are well known, but we sketch a plausibility argument to help guide the less familiar reader's understanding.
  First we consider the construction of the unit natural transformation $\upsilon:\id_{\cat^{\pP}} \Rightarrow f^*f_*$.
  Fix a $\pP$-module $M$, then by unraveling the definition of the pullback of the pushforward yields
  \begin{eqnarray*}
  f^*f_*M (p) & = & (f_*M)(f(p)) \\
    & = & \varinjlim_{p' \mid f(p')\leq f(p)} M(p').
  \end{eqnarray*}
  Now certainly it is the case that
  \[
    p\in \{p' \mid f(p')\leq f(p)\}
  \]
  because $f(p)\leq f(p)$.
  This implies that $p$ participates in the diagram that the colimit is taken over and hence there is a natural morphism
  \[
    M(p) \to \varinjlim_{p' \mid f(p')\leq f(p)} M(p').
  \]
  These piece together to form a morphism of modules
  $M \to f^*f_*M$ that is natural in $M$. This defines the
  unit of the adjunction $\upsilon$.

  To construct the co-unit of the adjunction
  $\chi: f_*f^* \Rightarrow \id_{\cat^{\pQ}}$, we follow
  a similar line of reasoning.
  Fix a $\pQ$-module $N$ and consider the following string of identies and morphisms:
  \begin{eqnarray*}
  f_*f^*N (q) &=& \varinjlim_{p \mid f(p)\leq q} (f^*N)(p)\\
    &=& \varinjlim_{p \mid f(p)\leq q} N(f(p))\\
    &\rightarrow& N(q)
  \end{eqnarray*}
  The last arrow exists by virtue of the fact that $N(q)$ has natural maps from all the elements that are in the sub-level set at $q$ and thus defines a co-cone.
  The colimit is the initial object in the category of co-cones, so it maps naturally to $N(q)$.
\end{proof}

\subsection{Cofinality}

Many arguments involving left Kan extensions, and hence colimits, requires showing that a particular functor is cofinal.
Cofinality allows us to replace one colimit with an equivalent, often simpler, colimit.

\begin{defn}\label{defn:cofinal}
A functor $E\colon \aat \to \bat$ is \define{cofinal} if for every object $b$ in $\bat$ the comma category $(b \downarrow E)$ is
\begin{itemize}
\item non-empty, and
\item connected.
\end{itemize}
Equivalently, a functor $E$ is cofinal if for every functor $F\colon\bat \to \cat$ to any category $\cat$ the induced map on colimits
\[
	\varinjlim F\circ E \to \varinjlim F
\]
is an isomorphism.
\end{defn}

\begin{rmk}
Note that the equivalence of these two definitions says that whether a diagram $F$ indexed by $\bat$ has the same colimit when restricted along $E\colon\aat \to \bat$ is dictated by the ``topological'' properties (nonemptiness and connectedness) of the comma categories $(b \downarrow E)$ for all objects $b$ in $\bat$.
Viewing these comma categories as fibers, the equivalence of the above two cofinality conditions is perhaps best viewed as a categorical analogue of the Vietoris Mapping Theorem.
\end{rmk}

\subsection{Cosheaves}

As an application of colimits and Kan extensions, we consider an important class of functors out of the open sets of a topological space.

\begin{defn}\label{defn:covers}
Let $\topX$ be a topological space.
We denote the poset of open sets in $\topX$, ordered by containment, by $\Open(\topX)$.
Consider a collection of open sets in $\topX$, which we write as $\covU=\{U_i\} \subseteq \Open(\topX)$.
\begin{enumerate}
	\item A \define{cover} of $U$ is a collection of open sets $\covU$ whose union is $U$.
	\item A \define{\v{C}ech cover} of $U$ is a cover $\covU$ of $U$ with the property that whenever a finite collection of $\{U_i\}_{\i\in \sigma}\subset \covU$ has non-empty intersection $U_{\sigma}=\cap_{i\in \sigma} U_i$, then $U_{\sigma}\in \covU$.
	\item A \define{basic cover} of $U$ is a cover $\covU$ of $U$ with the property that whenever $U_i, U_j\in \covU$, then $U_i\cap U_j$ is the union of elements in $\covU$.
\end{enumerate}
We note that every \v{C}ech cover is a basic cover.
\end{defn}

The notion of a basic cover comes from considering the defining properties of a basis for a topological space $\topX$---a basis is rarely closed under intersection, rather the intersections are unions of elements of the basis.

\begin{rmk}
The term \v{C}ech cover is borrowed from Dugger and Isaksen's article~\cite{dugger2004topological}.
The notion of a basic cover is closely related to the notion of a \emph{complete cover} given in the same article.
The difference is that a basic cover requires that pairwise intersections be covered, whereas a complete cover requires that all finite intersections be covered by elements of the cover.
For Dugger and Isaksen, this condition makes sense as they were interested in ``higher'' colimits whereas we are interested in ordinary colimits.
\end{rmk}

\begin{ex}[Intersections of Principals Not Principal]
Consider a down set $S\in \Down(\pP)$.
The collection of principal down sets $\{D_p\}_{p\in S}$ is a basic cover of the set $S$.
Note that the intersection of two principal down sets need not be principal in general, as the following example shows.
Here an element is higher in the partial order if an arrow points towards it.
\[
\begin{tikzcd}
& \bullet & \\
\bullet \ar[ru] \ar[rd] & \bullet \ar[l] \ar[r] \ar[u] \ar[d] & \bullet \ar[lu] \ar[ld] \\
& \bullet &
\end{tikzcd}
\]
If we take the principal down sets at the top and bottom vertex, then their intersection has two maximal elements, contained in the horizontal zig-zag.
\end{ex}

We now provide two notions of a cosheaf.

\begin{defn}\label{defn:cosheaf}
Let $\topX$ be a topological space and let $\cat$ be a category with all colimits.
A functor $\cosheaf{F}\colon\Open(\topX) \to \cat$ is a \define{cosheaf} if for every open set $U$ and every \v{C}ech cover $\calU$ of $U$ the universal arrow
\[
\varinjlim_{U_i\in\calU} \cosheaf{F} (U_i)
\rightarrow
\cosheaf{F}\left(\varinjlim_{U_i\in\calU} U_i\right)
=
\cosheaf{F} (\cup U_i)
=
\cosheaf{F}(U)
\]
is an isomorphism.

Similarly, a \define{basic cosheaf} is a functor $\cosheaf{F}\colon\Open(\topX) \to \cat$ with the property that for every open set $U$ and every basic cover of $U$, the same universal arrow above is an isomorphism.
Note that every basic cosheaf is a cosheaf by virtue of the fact that a \v{C}ech cover is a basic cover.
Denote the category of cosheaves on $\topX$ by $\CoShv(\topX;\cat)$ and the category of basic cosheaves by $\CoShv_{\flat}(\topX;\cat)$
\end{defn}

We are now in a position to state a main theorem: that the left Kan extension provides a basic cosheaf.
We refer to~\cite{curry2019kan} for a complete proof.

\begin{thm}
\label{thm:kan-cosheaves}

Let $M\colon\pP \to \cat$ be a functor from a poset $\pP$ to a co-complete category $\cat$, i.e. a $\pP$-module valued in $\cat$.
Let $\iota\colon \pP \to \Down(\pP)$ denote the functor that takes an element $p\in \pP$ to the principal down set $D_p$.
The left Kan extension of $M$ along $\iota$, written $\iota_*M$ below, is a basic cosheaf.
\[
\xymatrix{ \pP \ar[r]^{M} \ar[d]_{\iota} & \cat \\ \Down(\pP) \ar@{.>}[ru]_{\Lan_{\iota} M = \iota_*M} & }
\]
\end{thm}

\printbibliography

\end{document}